\documentclass[12pt,reqno]{amsart} 

\usepackage{times}
\usepackage[usenames, dvipsnames]{color}
\usepackage{todonotes}
\usepackage[
  margin=1in
]{geometry}

\usepackage{amssymb,amsfonts,amsthm}
\usepackage{comment}
\usepackage{enumitem}
\usepackage{graphicx}

\usepackage{accents}
\newcommand{\ubar}[1]{\underaccent{\bar}{#1}}

\usepackage[dvipsnames]{color}

\usepackage[colorlinks=true, pdfstartview=FitV, linkcolor=black, citecolor=blue, urlcolor=blue]{hyperref}

\def\Xint#1{\mathchoice
{\XXint\displaystyle\textstyle{#1}}%
{\XXint\textstyle\scriptstyle{#1}}%
{\XXint\scriptstyle\scriptscriptstyle{#1}}%
{\XXint\scriptscriptstyle%
\scriptscriptstyle{#1}}%
\!\int}
\def\XXint#1#2#3{{\setbox0=\hbox{$#1{#2#3}{%
\int}$ }
\vcenter{\hbox{$#2#3$ }}\kern-.6\wd0}}
\def\barint{\, \Xint -} 
\def\bariint{\barint_{} \kern-.4em \barint}
\def\bariiint{\bariint_{} \kern-.4em \barint}
\renewcommand{\iint}{\int_{}\kern-.34em \int} 
\renewcommand{\iiint}{\iint_{}\kern-.34em \int} 

\DeclareMathAlphabet{\mathcal}{OMS}{cmsy}{m}{n}

\theoremstyle{plain}

\newtheorem{theorem}{Theorem}[section]

\newtheorem{lemma}[theorem]{Lemma}

\newtheorem{corollary}[theorem]{Corollary}
\newtheorem{proposition}[theorem]{Proposition}

\theoremstyle{definition}
\newtheorem{remark}[theorem]{Remark}

\newcommand{\R}{\mathbb{R}}

\newcommand{\p}{\partial}

\newcommand{\les}{\lesssim}
\newcommand{\ges}{\gtrsim}
\newcommand{\norm}[1]{\lVert #1 \rVert}
\renewcommand{\:}{\colon}

\newcommand{\loc}{{\rm loc}}

\let\div\relax
\DeclareMathOperator{\div}{div}

\DeclareMathOperator{\sgn}{sgn}
\let\tilde\relas
\newcommand{\tilde}[1]{\widetilde{#1}}

\newcommand{\BMO}{{\rm BMO}}

\newcommand{\avg}{{\rm avg}}
\newcommand{\Vol}{{\rm Vol}}

\numberwithin{equation}{section}
\setlist[enumerate]{leftmargin=*}

\title{Regularity properties of passive scalars with rough divergence-free drifts}
\author{Dallas Albritton}
\address{Courant Institute of Mathematical Sciences, New York University, New York, NY 10012}
\email{daa399@cims.nyu.edu}
\author{Hongjie Dong}
\address{Division of Applied Mathematics, Brown University, 182 George Street, Providence, RI 02912}
\email{Hongjie\_Dong@brown.edu}

\date{\today}

\begin{document}
\begin{abstract}
We present sharp conditions on divergence-free drifts in Lebesgue spaces for the passive scalar advection-diffusion equation
\[
	\p_t \theta - \Delta \theta + b \cdot \nabla \theta = 0
\]
to satisfy local boundedness, a single-scale Harnack inequality, and upper bounds on fundamental solutions. We demonstrate these properties for drifts $b$ belonging to $L^q_t L^p_x$, where $\frac{2}{q} + \frac{n}{p} < 2$, or $L^p_x L^q_t$, where $\frac{3}{q} + \frac{n-1}{p} < 2$. For steady drifts, the condition reduces to $b \in L^{\frac{n-1}{2}+}$. The space $L^1_t L^\infty_x$ of drifts with `bounded total speed' is a borderline case and plays a special role in the theory. To demonstrate sharpness, we construct counterexamples whose goal is to transport anomalous singularities into the domain `before' they can be dissipated.
\end{abstract}


\maketitle
\setcounter{tocdepth}{1}
\tableofcontents

\section{Introduction}
We consider the linear advection-diffusion equation
\begin{equation}
\label{eq:ADE}
	\tag{A-D}
	\p_t \theta - \Delta \theta + b \cdot \nabla \theta = 0.
\end{equation}
The solution $\theta = \theta(x,t)$ is known as a \emph{passive scalar}, and the prescribed divergence-free velocity field $b = b(x,t)$ is known as the \emph{drift}.

Our goal is to understand, in detail, the regularity properties of~\eqref{eq:ADE} when the drift $b$ is \emph{rough}. Rough divergence-free drifts arises naturally in the context of nonlinear PDEs in fluid dynamics.

To understand what is `rough', we recall the scaling symmetry
\begin{equation}
	\label{eq:scalingsymmetry}
	u \to u(\lambda x, \lambda^2 t), \quad b \to \lambda b(\lambda x,\lambda^2 t), \quad \lambda >0.
\end{equation}
In dimensional analysis, one writes $[x] = L$, $[t] = L^2$, and $[b] = L^{-1}$. The scaling~\eqref{eq:scalingsymmetry} identifies the Lebesgue spaces $L^q_t L^p_x$, where $\frac{2}{q} + \frac{n}{p} \leq 1$, as \emph{(sub)critical} spaces for the drift, meaning spaces whose norms do not grow upon `zooming in' with the scaling symmetry. For example,
\begin{equation}
	X = L^\infty_t L^n_x,\, L^2_t L^\infty_x,\, L^{n+2}_{t,x}
\end{equation} are \emph{critical spaces}, whose norms are dimensionless, i.e., invariant under the symmetry~\eqref{eq:scalingsymmetry}. Here, $n \geq 2$ is the spatial dimension.

When $b$ belongs to one of the above critical Lebesgue spaces, it is not difficult to adapt the work of De Giorgi, Nash, and Moser~\cite{de1957sulla,nash1958continuity,moserharnack} to demonstrate that weak solutions of~\eqref{eq:ADE} are H{\"o}lder continuous and satisfy Harnack's inequality. The above threshold is known to be \emph{sharp} for H{\"o}lder continuity within the scale of Lebesgue spaces, see the illuminating counterexamples in~\cite{silvestre2013loss,wu2021supercritical}. The divergence-free condition moreover allows access to drifts in the critical spaces $X = L^\infty_t L^{-1,\infty}_x, L^\infty_t \BMO^{-1}_x$, considered by~\cite{osada1987,SSSZ,qianxisingular2019}. In these spaces, it is furthermore possible to prove Gaussian upper and lower bounds on fundamental solutions in the spirit of Aronson~\cite{aronson}.

In this paper, we are concerned with \emph{supercritical} drifts, for which continuity may fail. Nonetheless, much of the regularity theory may be salvaged. The divergence-free structure plays a crucial role here,\footnote{The divergence-free structure plays a more subtle role in the critical case. Without this structure, the drift is required to be small in a critical Lebesgue space or Kato class, and local boundedness may depend on the `profile' of the drift.} as is already visible from the computation
\begin{equation}
	\int (b \cdot \nabla \theta) \theta \phi^2 \, dx \, dt =  \int b \cdot \nabla \left( \frac{\theta^2}{2} \right) \phi^2 \, dx \, dt = - \int  \theta^2 (b \cdot \nabla \phi) \phi.
\end{equation}
With this well known observation, one may apply Moser's iteration scheme to demonstrate that, when $b \in L^q_t L^p_x$ and $\frac{2}{q} + \frac{n}{p} < 2$, solutions are \emph{locally bounded}, see~\cite{nazarov2012harnack}. Typical examples are
\begin{equation}
	X = L^\infty_t L^{\frac{n}{2}+}_x, L^{\frac{n+2}{2}+}_{t,x}.
\end{equation}
Under these conditions (and a weak background assumption), the Harnack inequality persists as a \emph{single-scale Harnack inequality}~\cite{ignatovaelliptic,ignatovakukavicaryzhik}: In the steady case $\theta = \theta(x)$,
\begin{equation}
	\sup_{B_R} \theta \leq C_R \inf_{B_R} \theta,
\end{equation}
where $C_R$ may become unbounded as $R \to 0^+$. Whereas a scale-invariant Harnack inequality implies H{\"o}lder continuity, it is perhaps less well known that a single-scale Harnack may hold in the absence of H{\"o}lder continuity. Finally, in this setting, pointwise upper bounds on fundamental solutions continue to hold, although they have `fat tails' compared to their Gaussian counterparts~\cite{zhang2004strong,qianxilelllq2019}. 

Our first contribution is to understand the sharpness of the condition $b \in L^q_t L^p_x$, $\frac{2}{q} + \frac{n}{p} < 2$. We find that $L^1_t L^\infty_x$, the space of drifts with `bounded total speed' (in the terminology of~\cite{taolocalizationcompactness}), plays a special role and informs the counterexamples we construct in Section~\ref{sec:counters}. We summarize the results pertaining to this condition in Theorem~\ref{thm:lqlp}. \\

In the steady case, there is an additional subtle feature, which is not well known and that we find surprising: Local boundedness continues to hold when $b \in L^{\frac{n-1}{2}+}$. To our knowledge, this `dimension reduction' was first observed in this context by Kontovourkis~\cite{kontovourkis2007elliptic} in his thesis.\footnote{
This `dimension reduction' itself goes back at least to work of Frehse and R$\overset{\circ}{\text{u}}${\v z}i{\v c}ka~\cite{frehse1996existence} on the steady Navier-Stokes equations in $n=6$. The `slicing' is also exploited by Struwe in~\cite{struwenavierstokes}.} Heuristically, Kontovourkis' key observation is as follows. Consider the basic $L^2$ energy estimate, in a ball $B_r$ without smooth cut-off. The drift contributes the boundary term
\begin{equation}
	\int_{B_r} (b \cdot \nabla \theta) \theta \, dx = \int_{\p B_r} \frac{\theta^2}{2} b \cdot n \, d\sigma,
\end{equation}
where $d\sigma$ is the surface area measure.
Since $\nabla \theta \in L^2(B_R)$, on `many slices' $r \in (R/2,R)$, we have $\nabla \theta \in L^2(\p B_r)$, with a quantitative bound. Similarly, $b$ belongs to $L^{\frac{n-1}{2}+}(\p B_r)$ on `many slices'. Thus, one may exploit Sobolev embedding on the sphere $\p B_r$ to estimate the boundary term. This dimension reduction was recently rediscovered by Bella and Sch{\"a}ffner~\cite{belleschaffnerscpam}, who proved local boundedness and a single-scale Harnack inequality in the context of certain degenerate elliptic PDEs, which we review below.

Since the work of Kontovourkis, it has been an interesting question, what dimension reduction holds in the parabolic setting? In particular, is $b \in L^{\frac{n+1}{2}+}_{t,x}$ enough for local boundedness? Very recently, X.~Zhang~\cite{zhang2020maximum} generalized the work~\cite{belleschaffnerscpam} of Bella and Sch{\"a}ffner to the parabolic setting, and among other things, demonstrated local boundedness under the condition $b \in L^p_x L^q_t$, $\frac{3}{q} + \frac{n-1}{p} < 2$, $p \leq q$, see Corollary~1.5 therein. Crucially, the order of integration is reversed. The condition $b \in L^{\frac{n-1}{2}+}_x L^\infty_t$ implies the elliptic case of Kontovourkis. From this condition, we see that, perhaps, one dimension is not `reduced', but rather hidden into the time variable.

In Theorem~\ref{thm:dimreductionversion}, we present our second contribution, namely, (i) the parabolic Harnack inequality and pointwise upper bounds on fundamental solutions in this setting, and (ii) counterexamples which illustrate the meaning and sharpness of the `dimension reduction'. \\

We now present the main results, which constitute a detailed picture of the local regularity theory for the passive scalar advection-diffusion equation~\eqref{eq:ADE} with supercritical drifts.

Let $n \geq 2$, $\Omega \subset \R^n$ be a bounded domain, and $\Omega' \subset\subset \Omega$ be a subdomain. Let $I = (S,T]$ and $I'=(S',T'] \subset I$ be finite intervals such that $S<S'$. Let $Q_I = \Omega \times I$ and $Q_I' = \Omega' \times I'$. Let $p,q \in [1,+\infty]$.

 \begin{theorem}[$b \in L^q_t L^p_x$]
 \label{thm:lqlp}
\emph{(Local boundedness)}~\cite{nazarov2012harnack} If
\begin{equation}
	\zeta_0 := \frac{2}{q} + \frac{n}{p} < 2,
\end{equation}
then we have the following \emph{quantitative local boundedness} property: If $\theta \in L^1(Q_I) \cap C^\infty(Q_I)$ satisfies the drift-diffusion equation~\eqref{eq:ADE} in $Q_I$
with divergence-free drift $b \in C^\infty(Q_I)$ and
\begin{equation}
	b \in L^q_t L^p_x(Q_I),
\end{equation}
 then
\begin{equation}
	\sup_{Q_I'} |\theta| \les \norm{\theta}_{L^1(Q_I)},
\end{equation}
where the implied constant depends on $n$, $\Omega$, $\Omega'$, $I$, $I'$, $p$, $q$, and $\norm{b}_{L^q_t L^p_x(Q_I)}$. \\

\emph{(Single-scale Harnack)}~\cite{ignatovakukavicaryzhik} If, additionally, $b \in L^2_t H^{-1}_x(Q_I)$ and $\theta > 0$, then we have the following \emph{quantitative Harnack inequality}: If $I_1, I_2 \subset\subset I$ are intervals satisfying $\sup I_1 < \inf I_2$, then
\begin{equation}
	\sup_{\Omega' \times I_1} \theta \les \inf_{\Omega' \times I_2} \theta,
\end{equation}
where the implied constant depends on $n$, $\Omega$, $\Omega'$, $I$, $I_1$, $I_2$, $p$, $q$, $\norm{b}_{L^q_t L^p_x(Q_I)}$, and $\norm{b}_{L^2_t H^{-1}_x(Q_I)}$. \\

\emph{(Bounded total speed)} If $(p,q) = (\infty,1)$, then the above quantitative local boundedness property holds with constants depending on $b$ itself rather than $\| b \|_{L^1_t L^\infty_x(Q_I)}$. (The property is false without this adjustment.) \\

 \emph{(Sharpness)} Let $Q = B_1 \times (0,1)$. There exist a smooth divergence-free drift $b \in C^\infty(Q)$ belonging to $L^q_t L^p_x(Q)$ for all $(p,q) \in [1,+\infty]^2$ with $2/q+n/p = 2$, $(p,q) \neq (\infty,1)$, and satisfying the following property. There exists a smooth solution $\theta \in L^\infty_t L^1_x \cap C^\infty(Q)$ to the advection-diffusion equation~\eqref{eq:ADE} in $Q$ with
\begin{equation*}
	 \sup_{B_{1/2} \times (0,T)} |\theta| \to +\infty \text{ as } T \to 1_-.
\end{equation*}
In particular, the above quantitative local boundedness property fails when $2/q+n/p=2$ and $q > 1$. \\

\emph{(Upper bounds on fundamental solutions)}~\cite{qianxilelllq2019} If the divergence-free drift $b \in C^\infty_0(\R^n \times [0,+\infty))$ belongs to $L^q_t L^p_x(\R^n \times \R_+)$ and $1 \leq \zeta_0 < 2$, then the fundamental solution $\Gamma = \Gamma(x,t;y,s)$ to the parabolic operator $L=\partial_t-\Delta+b \cdot \nabla$ satisfies
\begin{equation}
	\label{eq:upperboundslqlp}
\begin{aligned}
\Gamma(x,t;y,s) &\leq C (t-s) \left( \frac{1}{t-s} + \frac{M_0}{(t-s)^{1+\frac{\alpha_0}{2}}} \right)^{\frac{n+2}{2}}\\
&\quad \times \left[ \exp \left( - \frac{c |x-y|^2}{t-s}  \right) + \exp\left(  - \frac{ c |x-y|^{1+\frac{1}{1+\alpha_0}}}{((t-s) M_0)^{\frac{1}{1+\alpha_0}}} \right) \right]
\end{aligned}	
\end{equation}
for all $x,y \in \R^n$ and $0 \leq s < t < +\infty$. Here,
\begin{equation}
	\alpha_0 = \frac{\zeta_0-1}{\theta_2}, \quad M_0 = C \| b \|_{L^q_t L^p_x(\R^n \times \R_+)}^{\frac{1}{\theta_2}} + \frac{1}{4}, \quad \theta_2 = 1-\frac{\zeta_0}{2}.
\end{equation}
 \end{theorem}

 See Figure~\ref{fig:dim3} for an illustration of Theorem~\ref{thm:lqlp}.

	\begin{figure}[h!]
		\centering
		\includegraphics[width=4in]{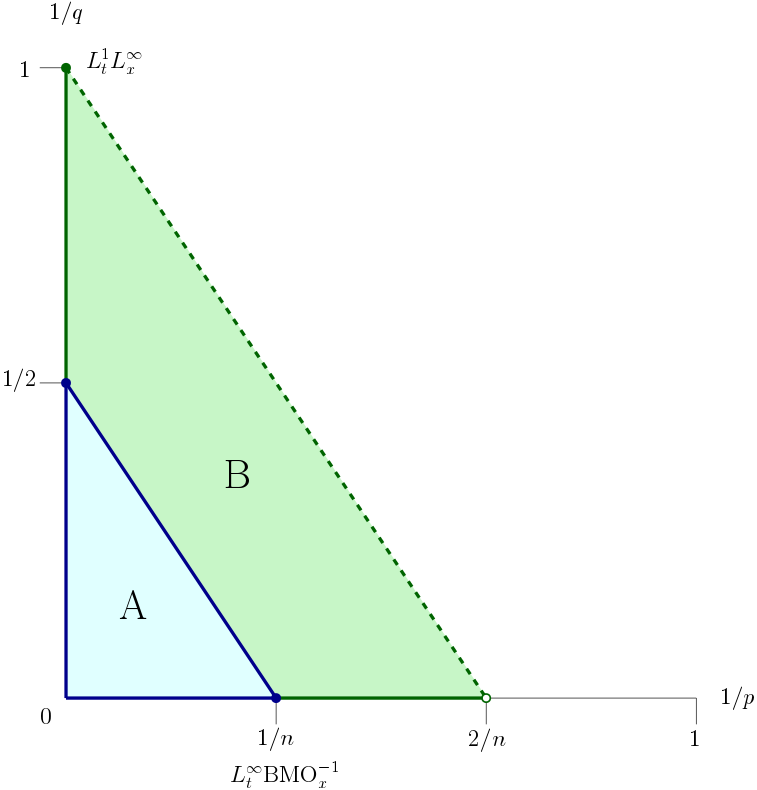}
	\caption{Divergence-free drift $b \in L^q_t L^p_x$ in dimension $n \geq 2$ (dimension $n=3$ illustrated above). \emph{Region~A} ($2/q + n/p \leq 1$): Local boundedness, Harnack inequality, and H{\"o}lder continuity. \emph{Region~B} ($1 < 2/q + n/p < 2$ or $(p,q) = (\infty,1)$): Local boundedness and single-scale Harnack inequality. \emph{Dashed line}: Local boundedness is false.}
\label{fig:dim3}
\end{figure}

\begin{theorem}[$b \in L^p_x L^q_t$]
\label{thm:dimreductionversion}
\emph{(Local boundedness)}~\cite{zhang2020maximum} If
\begin{equation}
	\zeta_0 := \frac{3}{q} + \frac{n-1}{p} < 2,\quad p\le q,
\end{equation}
then we have the following \emph{quantitative local boundedness} property: If $\theta \in L^1(Q_I) \cap C^\infty(Q_I)$ satisfies the drift-diffusion equation~\eqref{eq:ADE} in $Q_I$
with divergence-free drift $b \in C^\infty(Q_I)$ and
\begin{equation}
	b \in L^p_x L^q_t(Q_I),
\end{equation}
then
\begin{equation}
	\sup_{Q_I'} |\theta| \les \norm{\theta}_{L^1(Q_I)},
\end{equation}
where the implied constant depends on $n$, $\Omega$, $\Omega'$, $I$, $I'$, $p$, $q$, and $\norm{b}_{L^p_x L^q_t(Q_I)}$. \\

\emph{(Single-scale Harnack)} If, additionally, $b \in L^2_t H^{-1}_x(Q_I)$ and $\theta > 0$, then we have the following \emph{quantitative Harnack inequality}: If $I_1, I_2 \subset\subset I$ are intervals satisfying $\sup I_1 < \inf I_2$, then
\begin{equation}
	\sup_{\Omega' \times I_1} \theta \les \inf_{\Omega' \times I_2} \theta,
\end{equation}
where the implied constant depends on $n$, $\Omega$, $\Omega'$, $I$, $I_1$, $I_2$, $p$, $q$, $\norm{b}_{L^p_x L^q_t(Q_I)}$, and $\norm{b}_{L^2_t H^{-1}_x(Q_I)}$. \\

\emph{(Sharpness, steady case)}. Let $n \geq 3$. The quantitative local boundedness property fails for steady drifts $b \in L^{\frac{n-1}{2}}(B)$ and steady solutions $\theta$ in the ball $B$. \\

\emph{(Sharpness, time-dependent case)}. Let $n \geq 2$ and $Q = B_1 \times (0,1)$. There exist a smooth divergence-free drift $b \in C^\infty(Q)$ belonging to $L^p_x L^q_t(Q)$ for all $p,q \in [1,+\infty]$ with $p \leq q$ and $3/q+(n-1)/p > 2$ and satisfying the following property. There exists a smooth solution $\theta \in L^\infty_t L^1_x \cap C^\infty(Q)$ to the advection-diffusion equation~\eqref{eq:ADE} in $Q$ with
\begin{equation*}
	 \sup_{B_{1/2} \times (0,T)} |\theta| \to +\infty \text{ as } T \to 1_-.
\end{equation*}
In particular, the above quantitative local boundedness property fails when $3/q+(n-1)/p > 2$ and $p \leq q$. Finally, the drift additionally belongs to $L^q_t L^p_x(Q)$ for all $(p,q) \in [1,+\infty]^2$ with $2/q+n/p > 2$.
 \\


\emph{(Upper bounds on fundamental solutions)} If the divergence-free drift $b \in C^\infty_0(\R^n \times [0,+\infty))$ belongs to $L^p_x L^q_t(\R^n \times \R_+)$ and $1 \leq \zeta_0 < 2$, then the fundamental solution $\Gamma = \Gamma(x,t;y,s)$ to the parabolic operator $L=\partial_t-\Delta+b \cdot \nabla$ satisfies
\begin{equation}
	\label{eq:upperboundsdimred}
\begin{aligned}
\Gamma(x,t;y,s) &\leq C (t-s) \left( \frac{1}{t-s} + \frac{M_0}{(t-s)^{1+\frac{\alpha_0}{2}}} \right)^{\frac{n+2}{2}}\\
 &\quad\times\left[ \exp \left( - \frac{c |x-y|^2}{t-s}  \right) + \exp\left(  - \frac{ c |x-y|^{1+\frac{1}{1+\alpha_0}}}{((t-s) M_0)^{\frac{1}{1+\alpha_0}}} \right) \right]
\end{aligned}	
\end{equation}
for all $x,y \in \R^n$ and $0 \leq s < t < +\infty$. Here,
\begin{equation}
	\alpha_0 = \frac{\zeta_0-1}{\theta_2}, \quad M_0 = C \| b \|_{L^p_x L^q_t(\R^n \times \R_+)}^{\frac{1}{\theta_2}} + \frac{1}{4}, \quad \theta_2 = 1-\frac{\zeta_0}{2}.
\end{equation}
\end{theorem}

See Figures~\ref{fig:dim4} and~\ref{fig:dim2} for an illustration of Theorem~\ref{thm:dimreductionversion}.

	\begin{figure}[h!]
		\centering
		\includegraphics[width=4in]{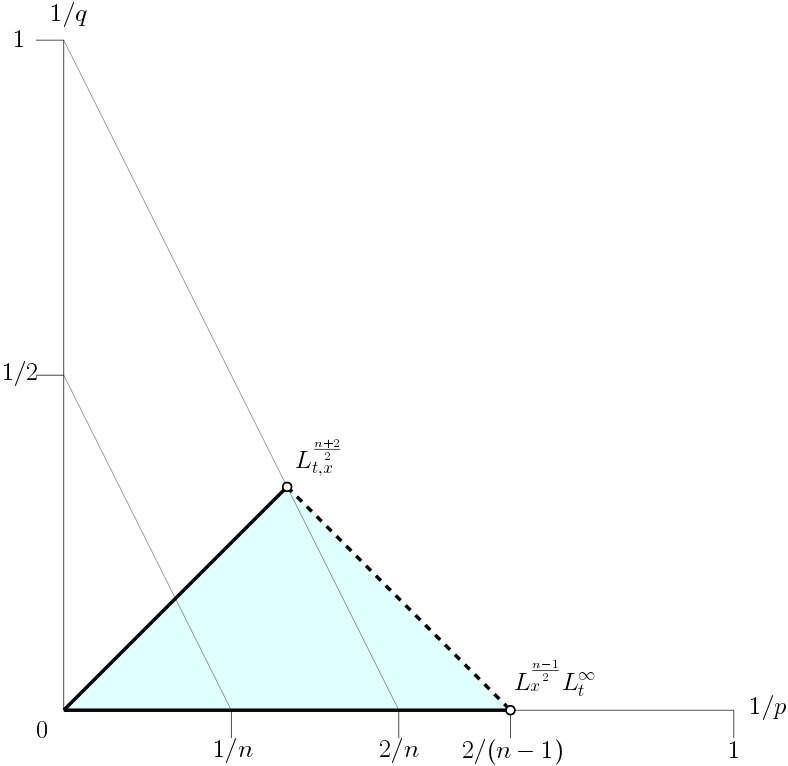}
	\caption{Divergence-free drift $b \in L^p_x L^q_t$, dimension $n \geq 3$ (dimension $n=4$ illustrated above). Local boundedness and single-scale Harnack inequality.}
\label{fig:dim4}
\end{figure}

	\begin{figure}[h!]
		\centering
		\includegraphics[width=4.1in]{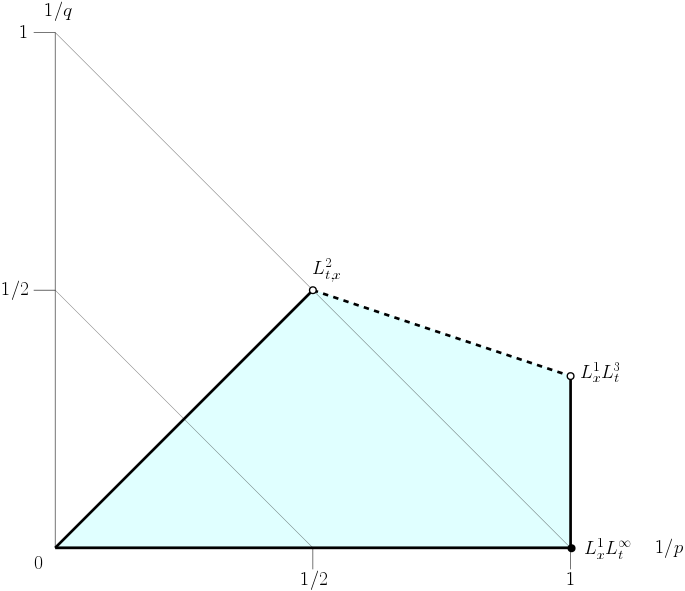}
	\caption{Divergence-free drift $b \in L^p_x L^q_t$, dimension $n =2$. Local boundedness and single-scale Harnack inequality.}
\label{fig:dim2}
\end{figure}

We prove Theorems~\ref{thm:lqlp} and~\ref{thm:dimreductionversion}, including the known results, in a self-contained way below.

\begin{remark}
The local boundedness property and Harnack inequality in Theorems~\ref{thm:lqlp} and~\ref{thm:dimreductionversion} can be easily extended to accomodate drifts satisfying $\div b \leq 0$ (with the background assumption $b \in L^2_{t,x}(Q_I)$ in the Harnack inequality). These properties and the fundamental solution estimates can also be extended to  divergence-form elliptic operators $\div a \nabla \cdot$ with bounded, uniformly elliptic~$a$.
\end{remark}

\begin{remark}
Our condition $b \in L^2_t H^{-1}_x(Q_I)$ appears to be new and is used to connect the forward- and backward-in-time regions in the Harnack inequality, see Lemma~\ref{lem:universal}. In contrast, the background condition in the single-scale Harnack inequality in~\cite{ignatovakukavicaryzhik} is $b \in L^\infty_t L^2_x(Q_I)$. 
\end{remark}

\subsection*{Discussion of dimension reduction principle}
The `slicing' described above in the steady setting is more subtle in the time-dependent setting because the anisotropic condition $\theta \in L^\infty_t L^2_x$ does not restrict well to slices in the radial variable $r$; compare this to the isotropic condition $\nabla \theta \in L^2_{t,x}$. Indeed, to `slice' in a variable, it seems necessary for that variable to be summed `last' (that is, on the outside) in the norm. The condition $b \in L^p_x L^q_t$, $p \leq q$, $\frac{3}{q} + \frac{n-1}{p} < 2$, in Theorem~\ref{thm:dimreductionversion} comes, roughly speaking, from interpolating between the isotropic condition $b \in L^{\frac{n+2}{2}+}_{t,x}$, in which the order of integration may be changed freely, and the dimensionally reduced condition $b \in L^{\frac{n-1}{2}+}_x L^\infty_t$, which implies that $b \in  L^\infty_t L^{\frac{n-1}{2}+}_\sigma(\p B_r \times I)$ on `many slices', say, a set of $r \in A \subset (1/2,1)$ with measure $|A| > 1/4$. Local boundedness under this condition was already observed by X. Zhang in~\cite[Corollary 1.5]{zhang2020maximum}, and the counterexamples we construct answer an open question in Remark~1.6 therein.

Our proof of local boundedness and the Harnack inequality is built on a certain \emph{form boundedness condition}~\eqref{eq:fbc}, see Section~\ref{sec:localbdd}, which subsumes a wide variety of possible assumptions on $b$.
 For example, in Proposition~\ref{pro:stuffissatisfied}, we verify~\eqref{eq:fbc} not only in the context of Theorems~\ref{thm:lqlp} and~\ref{thm:dimreductionversion} but also under the more general conditions
\begin{equation}
	b \in L^q_t L^\beta_r L^\gamma_\sigma((B_R \setminus B_{R/2}) \times I), \quad \beta \geq \frac{n}{2}, \quad \frac{2}{q} + \frac{1}{\beta} + \frac{n-1}{\gamma} < 2
\end{equation}
and
\begin{equation}
	b \in L^\kappa_r L^q_t L^p_\sigma((B_R \setminus B_{R/2}) \times I), \quad p \leq q, \quad \frac{3}{q} + \frac{n-1}{p} < 2.
\end{equation}
Furthermore, we allow arbitrarily low integrability $\kappa > 0$ in the radial variable; the slicing method does not require high integrability. Our proof of upper bounds on fundamental solutions is centered on a variant~\eqref{eq12.33} of the form boundedness condition, see Section~\ref{sec:upperboundsonfundamentalsolutions}, partially inspired by the work of Qi S. Zhang~\cite{zhang2004strong}.

We now describe the work~\cite{belleschaffnerscpam}, which was generalized to the parabolic setting in~\cite{zhang2020maximum}. The conditions in~\cite{belleschaffnerscpam} are on the ellipticity matrix $a$, which is allowed to be degenerate: Define
\begin{equation}
	\lambda(x) := \inf_{|\xi| = 1} \xi \cdot a(x) \xi, \quad \mu(x) := \sup_{|\xi| = 1} \frac{|a(x) \xi|^2}{\xi \cdot a(x)\xi}.
\end{equation}
If $n \geq 2$, $p,q \in (1,+\infty]$, and
\begin{equation}
	\label{eq:bellacondition}
	\lambda^{-1} \in L^q(B), \quad \mu \in L^p(B), \quad \frac{1}{p} + \frac{1}{q} < \frac{2}{n-1},
\end{equation}
then weak solutions of $- \div a \nabla u = 0$ are locally bounded and satisfy a single-scale Harnack inequality. The analogous condition with $\frac{2}{n}$ on the right-hand side is due to Trudinger in~\cite{trudinger}. By examples in~\cite{franchiseapioniserra}, the right-hand side cannot be improved to $\frac{2}{n-1} + \varepsilon$. Divergence-free drifts $b$ belong to the above framework: Under general conditions, it is possible to realize $b$ as the divergence of an antisymmetric \emph{stream matrix}: $b_i = d_{ij,i}$. Then we have $- \Delta \theta + b \cdot \nabla \theta = - \div [ (I+d) \nabla \theta ]$, and $\mu$ captures the antisymmetric part $d$. The steady examples we construct in Section~\ref{sec:counters} handle the equality case in~\eqref{eq:bellacondition}. We mention also the works~\cite{bellaschaffnerpqgrowth,bella2020non}.

Earlier, it was hoped that the dimension reduction could be further adapted to treat the case $b \in L^{\frac{n+1}{2}+}_{t,x}$ in the parabolic setting by estimating a half-derivative in time: $|\p_t|^{1/2} \theta \in L^2_{t,x}$, since this condition is better adapted to slicing than $\theta \in L^\infty_t L^2_x$. On the other hand, our counterexamples rule out this possibility. Half time derivatives in parabolic PDE go back, at least, to~\cite[Chapter III, Section 4]{ladyzhenskaya}, see~\cite{simonbortz} for further discussion.

\subsection*{Discussion of counterexamples and `bounded total speed'}
Solutions of~\eqref{eq:ADE} in the whole space and evolving from initial data $\theta_0 \in L^1(\R^n)$ become bounded instantaneously. This is captured by the famous Nash estimate~\cite{nash1958continuity}:
\begin{equation}
	\label{eq:nashestimate}
	\| \theta(\cdot,t) \|_{L^\infty(\R^n)} \les t^{-\frac{n}{2}} \| \theta_0 \|_{L^1(\R^n)},
\end{equation}
where the implied constant is \emph{independent} of the divergence-free drift $b$. The Nash estimate indicates that a divergence-free drift does not impede smoothing, in the sense of boundedness, of a density, even if the density is initially a Dirac mass. Therefore, for rough drifts, local boundedness must be violated in a different way: The danger is that the drift can `drag' an anomolous singularity into the domain of observation from outside. There is a competition between the drift, which transports the singularity with some speed, and the diffusion, which smooths the singularity at some rate. Will the singularity, entering from outside, be smoothed before it can be observed inside the domain?

Consider a Dirac mass $\delta_{x = - \vec{e_1}}$, which we seek to transport inside the domain. If one can transport the Dirac mass inside $B_{1/2}$ \emph{instantaneously}, one can violate local boundedness. This can be done easily via the drift $b(x,t) = \delta_{t = 0} \vec{e}_1$, which is singular in time. This example already demonstrates the importance of the space $L^1_t L^\infty_x$, whose drifts cannot transport the mass inside arbitrarily quickly.

To improve this example, we seek the most efficient way to transport the Dirac mass. Heuristically, the evolution of the Dirac mass is mostly supported in a ball of radius $R(t) \sim \sqrt{t}$. Therefore, we define our drift $b$ to be $S(t) \vec{e}_1$ restricted to this support. That is, the drift lives on a ball of radius $R(t)$ moving in the $x_1$-direction at speed $S(t)$. Since we wish to move the Dirac mass instantaneously, we guess that $S(t) \sim 1/t$. A back-of-the-envelope calculation gives
\begin{equation}
	\| b \|_{L^q_t L^p_x}^q \sim \int_0^1 S(t)^q R(t)^{\frac{nq}{p}} \, dt \sim \int_0^1 t^{-q+\frac{nq}{2p}} \, dt.
\end{equation}
The above quantity is finite when $2/q+n/p > 2$; more care is required to get the borderline cases in Theorem~\ref{thm:lqlp}, see Section~\ref{sec:counters}. This heuristic is the basis for our time-dependent counterexamples in Section~\ref{sec:counters}, except that we use appropriate subsolutions to keep the compact support property, we glue together many of these Dirac masses, and $S(t)$ must be chosen more carefully.

The above transport-vs.-smoothing phenomenon is also responsible for the `fat tails' in the upper bounds~\eqref{eq:upperboundslqlp} and~\eqref{eq:upperboundsdimred} on the fundamental solutions. These upper bounds do not align with the Nash estimate~\eqref{eq:nashestimate} because the Nash estimate does not effectively capture spatial localization.

The elliptic counterexample with $b \in L^{\frac{n-1}{2}}$ is achieved by introducing an ansatz which reduces the problem to counterexamples for the steady Schr{\"o}dinger equation $-\Delta u + Vu = 0$ in dimension $n-1$. These steady counterexamples are singular on a line through the domain, as they must be to respect the maximum principle.

The time-dependent counterexamples in $L^p_x L^q_t$ seem to be more subtle, and we only exhibit them in the non-borderline cases $\zeta_0 = \frac{3}{q} + \frac{n-1}{p} > 2$ and $p \leq q$. When $\zeta_0 = 2$, we have counterexamples in the cases $p = q = \frac{n+2}{2}$ and $(p,q) = (\frac{n-1}{2},+\infty)$ (the steady example). We believe that local boundedness fails also between these two points, but the counterexamples are yet to be exhibited, see Remark~\ref{rmk:anopenquestion}.

\subsection*{Review of existing literature}
Following the seminal works of De Giorgi~\cite{de1957sulla} and Nash~\cite{nash1958continuity}, Moser introduced his parabolic Harnack inequality~\cite{moserharnack,mosererrata} (see~\cite{moserelliptic} for the elliptic case), whose original proof relied on a parabolic generalization of the John-Nirenberg theorem concerning exponential integrability of $\BMO$ functions. Later, Moser published a simplified proof~\cite{moserpointwise}, whose basic methods we follow. In~\cite{SSSZ}, Seregin, Silvestre, {\v S}ver{\'a}k, and Zlato{\v s} generalized Moser's methods to accomodate drifts in $L^\infty_t \BMO^{-1}_x$. For recent work on boundary behavior in this setting, see~\cite{linhanjill,Hofmann2021}. Generalizations to critical Morrey spaces and the supercritical Lebesgue spaces are due to~\cite{nazarov2012harnack,ignatovaelliptic,ignatovakukavicaryzhik,ignatovaslightlysupercritical}.

The Gaussian estimates on fundamental solutions were discovered by Aronson~\cite{aronson} and were generalized to divergence-free drifts by Osada in~\cite{osada1987} ($L^\infty_t L^{-1,\infty}_x$) and Qian and Xi ($L^\infty_t \BMO^{-1}_x$) in~\cite{qianxisingular2019,qianxilelllq2019}. Important contributions are due to~\cite{zhang2004strong}, who developed Gaussian-like upper bounds in the supercritical case $b \in L^{\frac{n}{2}+}(\R^n)$, $n \geq 4$, and~\cite{liskevichzhang,milmansemenov,semenovregularitytheorems}, among others. For recent progress on Green's function estimates with sharp conditions on lower order terms, see~\cite{MR3787362,MR3941634, sakellaris2020scale,mourgoglou2019regularity}.

The primary examples concerning the regularity of solutions to~\eqref{eq:ADE} can be found in~\cite{SSSZ,silvestre2013loss,wu2021supercritical}. Counterexamples to continuity with time-dependent drifts can be constructed by colliding two discs of $+1$ (subsolution) and $-1$ (supersolution) with radii $R(t) \sim \sqrt{1-t}$ and speeds $S(t) \sim 1/\sqrt{1-t}$. The parabolic counterexamples with \emph{steady} velocity fields constructed therein are more challenging. See~\cite{Filonov2013,FilonovShilkin} for examples in the elliptic setting. We also mention Zhikov's counterexamples~\cite{zhikovuniqueness} to uniqueness when $b$ does not belong to $L^2$, whereas weak solutions with zero Dirichlet conditions are known to be unique when $b \in L^2$~\cite{zhang2004strong}.

For recent counterexamples in the regularity theory of parabolic systems based on self-similarity, see~\cite{mooney}. \\

\begin{remark}
At a technical level, there is a small but, perhaps, non-trivial gap in the proof of the weak Harnack inequality in~\cite{ignatovakukavicaryzhik}, see (3.22) therein, where it is claimed that $\log_+ (\theta/\mathbf{K})$ is a supersolution. This seems related to a step in the proof of Lemma 6.20, p. 124, in Lieberman's book~\cite{liebermanbook}, which we had difficulty following, see the first inequality therein. Both of these are related to improving the weak $L^1$ inequality.

We opt to follow Moser's proof in~\cite{moserpointwise} more directly and skip the weak Harnack inequality. In principle, one could directly apply the parabolic John-Nirenberg inequality in~\cite{moserharnack,fabesgarofalo} to obtain the weak Harnack inequality.
\end{remark}

\section{Local boundedness and Harnack's inequality}
\label{sec:localbdd}

Let $b$ be a smooth, divergence-free vector field defined on $B_{R_0} \times I_0$, where $R_0 > 0$ and $I_0$ is an open interval. In the sequel, we will use a \emph{form boundedness condition}, which we denote by~\eqref{eq:fbc}: 
\begin{quote}
There exist constants $M, N, \alpha> 0$, $\varepsilon \in [0,1/2)$, and $\delta \in (0,1]$ satisfying the following property. For every $R \in [R_0/2,R_0]$, subinterval $I \subset I_0$, and Lipschitz $u \in W^{1,\infty}(B_R \times I)$, there exists a measurable set $A = A(\varrho,R,I,u) \subset (\varrho,R)$ with $|A| \geq \delta (R-\varrho)$ and satisfying
\begin{equation}
	\label{eq:fbc}
	\tag{FBC}
	\begin{aligned}
	&- \frac{1}{|A|} \iint_{B_A \times I} \frac{|u|^2}{2} (b \cdot n) \, dx \, dt \leq \frac{MR_0^\alpha }{\delta^\alpha R_0^2 (R-\varrho)^\alpha} \iint_{(B_R \setminus B_\varrho) \times I} |u|^2 \, dx \, dt \\
	&\quad\quad +N \iint_{(B_R \setminus B_\varrho) \times I} |\nabla u|^2 \, dx \, dt + \varepsilon \sup_{t \in I} \int_{B_R} |u(x,t)|^2 \, dx,
	\end{aligned}
\end{equation}
where $B_A = \cup_{r \in A} \p B_r$ and $n$ is the outer unit normal.
\end{quote}

The LHS of~\eqref{eq:fbc} appears on the RHS of the energy estimates.


In the situations we consider, $M$ may depend on $R_0$, and we can predict its dependence based on dimensional analysis. For example, since $b$ has dimensions of $L^{-1}$, the quantity $$
R_0^{1-\frac{2}{q}-\frac{n}{p}} \| b \|_{L^q_t L^p_x(B_{R_0} \times R_0^2 I)}
$$
is dimensionless.

In Proposition~\ref{pro:stuffissatisfied}, we show that~\eqref{eq:fbc} is satisfied under the hypotheses of Theorems~\ref{thm:lqlp} and~\ref{thm:dimreductionversion}. \\

\emph{Notation}. In this section, $R_0/2 \leq \varrho < R \leq R_0$ and $-\infty < T < \tau < 0$. Let us introduce the backward parabolic cylinders $Q_{R,T} = B_R \times (T,0)$. Our working assumptions are that $\theta$ is a non-negative Lipschitz function and $b$ is a smooth, divergence-free vector field. To give precise constants, we will frequently use the notation
\begin{equation}
	\label{eq:Cdef}
	\mathbf{C}(\varrho,\tau,R,T,M,\delta,\alpha) = \frac{1}{\delta^2 (R-r)^2} + \frac{MR_0^\alpha}{\delta^\alpha R_0^2 (R-r)^\alpha} + \frac{1}{\tau-T}
\end{equation}
involving the various parameters from~\eqref{eq:fbc}. Our convention throughout the paper is that all implied constants may depend on $n$.

\begin{theorem}[Local boundedness]
	\label{thm:localboundedness}
	Let $\theta$ be a non-negative Lipschitz subsolution and $b$ satisfy~\eqref{eq:fbc} on $Q_{R,T}$. Then, for all $\gamma \in (0,2]$,
	\begin{equation}
	\sup_{Q_{\varrho,\tau}} \theta \les_{N,\alpha,\varepsilon,\gamma} \mathbf{C}^{\frac{n+2}{2\gamma}} \norm{\theta}_{L^\gamma(Q_{R,T} \setminus Q_{\varrho,\tau})}.
	\end{equation}
\end{theorem}

\begin{theorem}[Harnack inequality]
\label{thm:Harnack}
	Let $\theta$ be a non-negative Lipschitz solution on $Q^* = B \times (-T^*,T^*)$. Let $b \in L^2_t H^{-1}_x(Q^*;\R^n)$ satisfying~\eqref{eq:fbc} on $Q^*$. 
	Let $0 < \ell < T^*$ be the time lag. Then
	\begin{equation}
	\sup_{B_{1/2} \times (-T^*+\ell,0)} \theta \les_{N,M,A,T^*,\delta,\alpha,\varepsilon,\ell} \inf_{B_{1/2} \times (\ell,T^*)} \theta,
	\end{equation}
	where $A = \norm{b}_{L^2_t H^{-1}_x(Q^*;\R^n)}^2$.
\end{theorem}

\subsection{Verifying~\eqref{eq:fbc}}

We verify that~\eqref{eq:fbc} is satisfied in the setting of the main theorems:
\begin{proposition}[Verifying FBC]
\label{pro:stuffissatisfied}
Let $p,q,\beta,\gamma \in [1,+\infty]$, $\kappa \in (0,+\infty]$, and $b$ be a smooth, divergence-free vector field defined on $B_{R_0} \times I_0$.

1. If
\begin{equation}
	\label{eq:condition1}
	b \in L^q_t L^\beta_r L^\gamma_\sigma((B_{R_0} \setminus B_{R_0/2}) \times I_0), \quad \beta \geq \frac{n}{2}, \quad \zeta_0 := \frac{2}{q} + \frac{1}{\beta} + \frac{n-1}{\gamma} < 2,
\end{equation}
then $b$ satisfies~\eqref{eq:fbc} with $\alpha = 1/\theta_2$, $A = (\varrho,R)$, $\delta = 1$, $N = \varepsilon = 1/4$, and
\begin{equation}
	M = C \left( R_0^{1-\zeta_0} \| b \|_{L^q_t L^\beta_r L^\gamma_\sigma((B_{R_0} \setminus B_{R_0/2}) \times I_0)} \right)^{\frac{1}{\theta_2}} + \frac{1}{4},
\end{equation}
where $\theta_2 = 1-\zeta_0/2$.

2. If
\begin{equation}
	\label{eq:condition2}
	b \in L^\kappa_r L^q_t L^p_\sigma((B_{R_0} \setminus B_{R_0/2}) \times I_0), \quad p \leq q, \quad \zeta_0 := \frac{3}{q} + \frac{n-1}{p} < 2,
\end{equation}
then $b$ satisfies~\eqref{eq:fbc} with $\alpha = (1/\kappa + (q-1)/q)/\theta_2$, $\delta = 1/2$, $N = \varepsilon = 1/4$, and
\begin{equation}
	M = C \left( R_0^{1-\zeta_0} \| b \|_{L^\kappa_r L^q_t L^p_\sigma((B_{R_0} \setminus B_{R_0/2}) \times I_0)}\right)^{\frac{1}{\theta_2}} + \frac{1}{4},
\end{equation}
where $\theta_2 = 1-\zeta_0/2$.
\end{proposition}

\begin{corollary}[FBC in $L^p_x L^q_t$]
Let $p,q \in [1,+\infty]$ and $b$ be as above. If
\begin{equation}
	\label{eq:condition3}
	b \in L^p_x L^q_t((B_{R_0} \setminus B_{R_0/2}) \times I_0), \quad p \leq q, \quad \zeta_0 := \frac{3}{q} + \frac{n-1}{p} < 2,
\end{equation}
then $b$ satisfies~\eqref{eq:fbc} with $\alpha = (1/p + (q-1)/q)/\theta_2$, $\delta = 1/2$, $N = \varepsilon = 1/4$, and
\begin{equation}
	M = C \left( R_0^{1-\zeta_0} \| b \|_{L^p_x L^q_t((B_{R_0} \setminus B_{R_0/2}) \times I_0)}\right)^{\frac{1}{\theta_2}} + \frac{1}{4},
\end{equation}
where $\theta_2 = 1-\zeta_0/2$.
\end{corollary}

By Minkowski's inequality,~\eqref{eq:condition3} is a special case of~\eqref{eq:condition2} with $\kappa=p$.

\begin{remark}
\label{rmk:remarkaboutendpoints}
Condition~\eqref{eq:condition1} automatically enforces $q \in (1,+\infty]$ and $\gamma \in \left( \frac{n-1}{2},+\infty \right]$. Condition~\eqref{eq:condition2} automatically enforces $q \in \left( \frac{n+2}{2} ,+\infty \right]$ and $p \in \left( \frac{n-1}{2}, +\infty \right]$.
\end{remark}

\begin{proof}[Proof of Proposition~\ref{pro:stuffissatisfied}]
First, we rescale $R_0 = 1$. Let $1/2 \leq \varrho < R \leq 1$, $I \subset I_0$. We restrict to $n \geq 3$ and summarize $n=2$ afterward. Unless stated otherwise, the norms below are on $(B_R \setminus B_\varrho) \times I$.

\textit{1. Summary of embeddings for $u$}. By the Sobolev embedding theorem, we have
\begin{equation}
	\| u \|_{L^2_t L^{2^*_n}_x} \les \| u \|_{L^2} + \| \nabla u \|_{L^2},\quad 2^*_n=\frac{2n}{n-2}.
\end{equation}
After interpolating with $\| u \|_{L^\infty_t L^2_x}$ and $\| u \|_{L^2_{t,x}}$, we have
\begin{equation}
	\label{eq:firstinterpolation}
	\| u \|_{L^{q_1}_t L^{p_1}_x} \les \| u \|_{L^2}^{\theta_1} (\| u \|_{L^2} + \| \nabla u \|_{L^2} + \| u \|_{L^\infty_t L^2_x})^{1-\theta_1}
\end{equation}
for suitable $\theta_1 = \theta_1(q_1,p_1) \in (0,1)$ whenever
\begin{equation}
	\frac{2}{q_1} + \frac{n}{p_1} > \frac{n}{2}, \quad q_1 \in [2,+\infty), \quad p_1 \in [2,2^*_n).
\end{equation}

Next, we employ the Sobolev embedding theorem on the spheres $\p B_r$, $r \in (\varrho,R)$:
\begin{equation}
	\label{eq:sobolevembeddingspheres}
	\| u \|_{L^2_t L^2_r L^{2^*_{n-1}}_\sigma} \les \| u \|_{L^2} + \| \nabla u \|_{L^2}.
	\end{equation}
Interpolating with~\eqref{eq:firstinterpolation}, we have
\begin{equation}
	\label{eq:uinterpineq}
	\| u \|_{L^{q_2}_t L^{\beta_2}_r L^{\gamma_2}_\sigma} \les \| u \|_{L^2}^{\theta_2} (\| u \|_{L^2} + \| \nabla u \|_{L^2} + \| u \|_{L^\infty_t L^2_x})^{1-\theta_2}
\end{equation}
for suitable $\theta_2 = \theta_2(q_2,\beta_2,\gamma_2) \in (0,1)$ whenever
\begin{equation}
	\label{eq:q2belonging}
	\frac{2}{q_2} + \frac{1}{\beta_2} + \frac{n-1}{\gamma_2} > \frac{n}{2}, \quad q_2 \in [2,+\infty), \quad \beta_2 \in [2,2^*_n), \quad \gamma_2 \in [2,2^*_{n-1}).
\end{equation}
The condition~\eqref{eq:q2belonging} describes a region in the parameter space $\{ (1/q_2,1/\beta_2,1/\gamma_2) \} \subset \R^3$ in a tetrahedron with vertices $(1/2,1/2,1/2)$, $(1/2,1/2^*_n,1/2^*_n)$, $(1/2,1/2,1/2^*_{n-1})$, and $(0,1/2,1/2)$. We compute $\theta_2$ according to
\begin{equation}
	\label{eq:computetheta2}
	\theta_2 \left( \frac{n+2}{2} \right) + (1-\theta_2) \left(  \frac{n+2}{2} - 1 \right) = \frac{2}{q_2} + \frac{1}{\beta_2} + \frac{n-1}{\gamma_2}.
\end{equation}

\textit{2. Verifying~\eqref{eq:fbc} for condition~\eqref{eq:condition1}}. Choose $q_2/2 = q'$ and $\beta_2/2 = \beta'$, where $'$ denotes H{\"o}lder conjugates. This is admissible according to the restrictions on $q$, $\beta$, $q_2$, and $\beta_2$ described in~\eqref{eq:condition1}, Remark~\ref{rmk:remarkaboutendpoints}, and~\eqref{eq:q2belonging}. We choose $\gamma_2$ to satisfy
\begin{equation}
	\label{eq:thingtounfold}
	\left( \frac{2}{q} + \frac{1}{\beta} + \frac{n-1}{\gamma} \right) + 2 \left( \frac{2}{q_2} + \frac{1}{\beta_2} + \frac{n-1}{\gamma_2} \right) = n + 2.
\end{equation}
This is possible according to the numerology in~\eqref{eq:condition1} and~\eqref{eq:q2belonging}. Unfolding~\eqref{eq:thingtounfold}, we find $\gamma_2/2 = \gamma'$.  Finally, we choose $A = (\varrho,R)$ and compute $\theta_2$ according to~\eqref{eq:computetheta2}. Then
\begin{equation}
	\label{eq:computingforcondition1}
\begin{aligned}
	& \left| \frac{1}{|A|} \iint_{B_A \times I} \frac{|u|^2}{2} b \cdot n \, dx \, dt \right| \leq \frac{1}{R - \varrho} \| b \|_{L^q_t L^{\beta}_r L^{\gamma}_\sigma} \| u \|_{L^{q_2}_t L^{\beta_2}_r L^{\gamma_2}_\sigma}^2 \\
	&\quad\overset{\eqref{eq:uinterpineq}}{\leq} \frac{C}{R - \varrho} \| b \|_{L^q_t L^{\beta}_r L^{\gamma}_\sigma} \| u \|_{L^2}^{2\theta_2} (\| u \|_{L^2} + \| \nabla u \|_{L^2} + \| u \|_{L^\infty_t L^2_x})^{2-2\theta_2} \\
	&\quad\quad\leq C \frac{\| b \|_{L^q_t L^{\beta}_r L^{\gamma}_\sigma}^{\frac{1}{\theta_2}} }{(R-\varrho)^{\frac{1}{\theta_2}}} \| u \|_{L^2}^2 + \frac{1}{4} ( \| u \|_{L^2}^2 + \| \nabla u \|_{L^2}^2 + \| u \|_{L^\infty_t L^2_x}^2 ),
	\end{aligned}
\end{equation}
where we used Young's inequality in the last step. This implies~\eqref{eq:fbc}.

\textit{3. Verifying~\eqref{eq:fbc} for condition~\eqref{eq:condition2}}. First, we identify \emph{good slices} for $b$. Specifically, we apply Chebyshev's inequality in $r$ to the integrable function
	\begin{equation}
	r \mapsto \big\lVert b|_{I \times \p B_r} \big\rVert_{L^q_t L^p_\sigma (I \times \p B_r)}^\kappa
\end{equation}
to obtain that, on a set $A = A(\varrho,R,I)$ of measure $|A| \geq (R-\varrho)/2$, we have
\begin{equation}
	\label{eq:postcheby}
	\norm{b}_{L^\infty_r L^q_t L^p_\sigma(B_{A} \times I)} \les \frac{1}{(R-\varrho)^{\frac{1}{\kappa}}} \norm{b}_{L^\kappa_r L^q_t L^p_\sigma((B_R \setminus B_\varrho) \times I)}
\end{equation}
where $B_A = \cup_{r \in A} \p B_r$.

Now we choose $q_2/2 = q'$, $\beta_2 = q_2$, and $\gamma_2 \in [2,2^*_{n-1})$ satisfying~\eqref{eq:q2belonging} and
\begin{equation}
	\label{eq:thingtounfold2}
	\left( \frac{3}{q} + \frac{n-1}{p} \right) + 2 \left( \frac{3}{q_2} + \frac{n-1}{\gamma_2} \right) = n + 2.
\end{equation}
Unfolding~\eqref{eq:thingtounfold2}, we discover $\gamma_2/2 = p'$. This allows us to compute $\theta_2$ according to~\eqref{eq:computetheta2}. Moreover, since $q_2 = \beta_2$ in~\eqref{eq:q2belonging}, we have $L^{q_2}_t L^{\beta_2}_r L^{\gamma_2}_\sigma =  L^{\beta_2}_r L^{q_2}_t L^{\gamma_2}_\sigma$. Then
\begin{equation}
	\label{eq:computingforcondition2}
\begin{aligned}
	& \left| \frac{1}{|A|} \iint_{B_A \times I} \frac{|u|^2}{2} b \cdot n \, dx \, dt \right| \leq \frac{C}{R-\varrho} \| b \|_{L^\infty_r L^q_t  L^{p}_\sigma} (R - \varrho)^{1 - \frac{2}{\beta_2}} \| u \|_{L^{\beta_2}_r L^{q_2}_t  L^{\gamma_2}_\sigma}^2 \\
	&\quad \overset{\eqref{eq:uinterpineq}}{\leq}  \frac{C}{(R-\varrho)^{\frac{1}{\kappa} + \frac{2}{\beta_2}} }\norm{b}_{L^\kappa_r L^q_t L^p_\sigma((B_R \setminus B_\varrho) \times I)} \| u \|_{L^2}^{2 \theta_2} (\| u \|_{L^2} + \| \nabla u \|_{L^2} + \| u \|_{L^\infty_t L^2_x})^{2(1-\theta_2)} \\
	&\quad\quad \leq C \frac{\norm{b}_{L^\kappa_r L^q_t L^p_\sigma((B_1 \setminus B_{1/2}) \times I)}^{\frac{1}{\theta_2}}}{(R-\varrho)^{\frac{1}{\theta_2} (\frac{1}{\kappa} + \frac{2}{\beta_2})}} \| u \|_{L^2}^2 +  \frac{1}{4} ( \| u \|_{L^2}^2 + \| \nabla u \|_{L^2}^2 + \| u \|_{L^\infty_t L^2_x}^2 ).
	\end{aligned}
\end{equation}
This completes the proof of~\eqref{eq:fbc}.

\textit{4. Dimension $n=2$}. The Sobolev embedding~\eqref{eq:sobolevembeddingspheres} on the sphere bounds instead 
\begin{equation*}
	 \| u \|_{L^2_t L^2_r C^{1/2}_\sigma} \les \| u \|_{L^2} + \| \nabla u \|_{L^2},
\end{equation*}
so we must adjust the proof of the interpolation~\eqref{eq:uinterpineq}. After the initial interpolation step between $L^2_t L^{2^*_n}_x$, $L^\infty_t L^2_x$ and $L^2_{t,x}$ (where $2^*_n$ is now any large finite number), we apply the following Gagliardo-Nirenberg inequality on spheres:
\begin{equation}
	\label{eq:gagnironsphere}
	\| u \|_{L^q(B_r)} \les \| u \|_{L^p(B_r)}^\theta \| u \|_{C^{1/2}(B_r)}^{1-\theta}, \quad - \frac{n-1}{q} = \theta \left( - \frac{n-1}{p} \right) + \frac{1-\theta }{2},
\end{equation}
where  $p,q \in [1,+\infty]$, $\theta \in (0,1]$, and $r \in [1/2,1]$. This allows us to recover~\eqref{eq:uinterpineq}, where now $\gamma_2 = +\infty$ is also allowed, and complete the proof.

To prove~\eqref{eq:gagnironsphere}, we first use local coordinates on the sphere and a partition of unity\footnote{Alternatively, since $n=2$, we could argue on the flat torus without a partition of unity. The argument we present here is more general.} to reduce to functions $f$ on $\R^{n-1}$. Next, we use that $L^p \subset B^{-\frac{n-1}{p}}_{\infty,\infty}$, $C^{1/2} = B^{1/2}_{\infty,\infty}$, and real interpolation
\begin{equation}
	[B^{-\frac{n-1}{p}}_{\infty,\infty},B^{1/2}_{\infty,\infty}]_{\theta,1} = B^0_{q,1} \subset L^q
\end{equation}
to demonstrate
\begin{equation}
	\| f \|_{L^q} \les \| f \|_{L^p(B_r)}^\theta \| f \|_{C^{1/2}(B_r)}^{1-\theta} \les \varepsilon^{\frac{1}{\theta}} \| f \|_{L^p(B_r)} + \varepsilon^{-\frac{1}{1-\theta}} \| f \|_{C^{1/2}(B_r)}.
\end{equation}
We piece together $u$ from the functions $f$ and optimize in $\varepsilon$ to obtain~\eqref{eq:gagnironsphere}.
\end{proof}

\begin{remark}[On $\tilde{\text{FBC}}$]
\label{rmk:trackingparametervarepsilon0}
For pointwise upper bounds on fundamental solutions in Section~\ref{sec:upperboundsonfundamentalsolutions}, we require a global, revised form boundedness condition~\eqref{eq12.33}, in which we allow $\varepsilon$ to vary in $(0,+\infty)$. We refine~\eqref{eq:computingforcondition1} by keeping track of the $\varepsilon$ dependence in Hardy's inequality:
\begin{equation}
\begin{aligned}
	& \left| \frac{1}{|A|} \iint_{B_A \times I} \frac{|u|^2}{2} b \cdot n \, dx \, dt \right| \leq \frac{C \| b \|_{L^q_t L^{\beta}_r L^{\gamma}_\sigma((B_R \setminus B_\varrho) \times I)}^{\frac{1}{\theta_2}} }{\varepsilon^{\frac{1-\theta_2}{\theta_2}} (R-\varrho)^{\frac{1}{\theta_2}}} \| u \|_{L^2}^2 + \varepsilon \left( \| u \|_{L^2}^2 + \| \nabla u \|_{L^2}^2 + \| u \|_{L^\infty_t L^2_x}^2 \right),
	\end{aligned}
\end{equation}
A similar refinement holds in~\eqref{eq:computingforcondition2}. We require these refinements in the justification of Lemma~\ref{lem:verifyingtildefbc}.
\end{remark}

\subsection{Proof of local boundedness}

To begin, we prove Cacciopoli's inequality:
\begin{lemma}[Cacciopoli inequality]
	\label{lem:cacciopoli}
	Under the hypotheses of Theorem~\ref{thm:localboundedness},
	\begin{equation}
	\label{eq:cacciopoliinequality}
		\sup_{t \in (\tau,0)} \int_{B_r} |\theta(x,t)|^2 \, dx + \iint_{Q_{\varrho,\tau}} |\nabla \theta|^2 \, dx\, dt \les_{N,\alpha,\varepsilon} \mathbf{C} \iint_{Q_{R,T} \setminus Q_{\varrho,\tau}} |\theta|^2 \, dx\, dt.
\end{equation}
\end{lemma}

\begin{proof}[Proof of Lemma~\ref{lem:cacciopoli}]
 Let $\eta \in C^\infty_0(T,+\infty)$ satisfying $0 \leq \eta \leq 1$ on $\R$, $\eta \equiv 1$ on $(\tau,+\infty)$ and $0 \leq d\eta/dt \les 1/(\tau-T)$. Let $r \in (\varrho,R)$ and $t \in (\tau,0)$. To begin, we multiply by $\theta \eta^2$ and integrate over $B_r \times (T,t)$:
\begin{equation}
\begin{aligned}
	&\frac{1}{2} \int_{B_r} |\theta(x,t)|^2 \, dx + \iint_{B_r \times (\tau,t)} |\nabla \theta|^2 \, dx\, ds \\
	&\quad \leq \iint_{B_R \times (T,t)} |\theta|^2 \frac{d\eta}{ds} \eta \, dx\, ds
	+ \iint_{\p B_r \times (T,t)} \Big(\frac{d\theta}{dn} \theta \eta^2 - \frac{\theta^2}{2} (b \cdot n) \eta^2 \Big)\, d\sigma \, ds.
	\end{aligned}
\end{equation}
Next, we average in the $r$ variable over the set of `good slices', $A = A(\varrho,R,(T,t),\theta \eta)$, which was defined in~\eqref{eq:fbc}:
\begin{equation}
\begin{aligned}
	&\frac{1}{2} \int_{B_\varrho} |\theta(x,t)|^2 \, dx + \iint_{B_\varrho \times (\tau,t)} |\nabla \theta|^2 \, dx\, ds \\
	&\quad \leq \frac{C}{\tau-T} \iint_{B_R \times (T,\tau)} |\theta|^2 \, dx\, ds + \frac{1}{|A|} \iint_{B_A \times (T,t)}\Big( \frac{d\theta}{dn} \theta \eta^2 - \frac{\theta^2}{2} (b \cdot n) \eta^2\Big) \, dx\, ds,
	\end{aligned}
\end{equation}
where $B_A = \cup_{r \in A} \p B_{r}$.
Let us estimate the term containing $d\theta/dn$:
\begin{equation}
	\begin{aligned}
	&\frac{1}{|A|} \iint_{B_A \times (T,t)} \frac{d\theta}{dn} \theta  \eta^2 \, dx\, ds \\
	&\quad \leq \frac{1}{\delta^2 (R-r)^2} \iint_{Q_{R,T} \setminus Q_{\varrho,\tau}} |\theta|^2 \, dx\, ds + \iint_{Q_{R,T} \setminus Q_{\varrho,\tau}} |\nabla \theta|^2 \, dx\, ds.
	\end{aligned}
\end{equation}
To estimate the term containing $b$, we use~\eqref{eq:fbc} with $u=\theta \eta$:
\begin{equation}
\begin{aligned}
	&- \frac{1}{|A|} \iint_{B_A \times (T,t)} \frac{\theta^2}{2} (b \cdot n) \eta^2 \, dx\, ds \leq \frac{M}{\delta^\alpha (R-r)^\alpha} \iint_{Q_{R,T} \setminus Q_{\varrho,\tau}} |\theta|^2 \, dx\, ds\\
	&\quad  + N \iint_{Q_{R,T} \setminus Q_{\varrho,\tau}} |\nabla \theta|^2 \, dx\, ds + \varepsilon \sup_{s\in(T,t)} \int_{B_R} |\theta(x,s)|^2 \, dx.
	\end{aligned}
\end{equation}
Combining everything and applying $\sup_{t \in (T,0)}$, we obtain 
\begin{equation}
	\label{eq:combiningeverything}
\begin{aligned}
	&\frac{1}{2} \sup_{t \in (\tau,0)} \int_{B_\varrho} |\theta(x,t)|^2 \, dx + \iint_{Q_{\varrho,\tau}} |\nabla \theta|^2  \, dx\, dt \leq C \times \mathbf{C} \iint_{Q_{R,T} \setminus Q_{\varrho,\tau} } |\theta|^2 \, dx\, dt \\
	&\quad\quad + (1 + N) \iint_{Q_{R,T} \setminus Q_{\varrho,\tau}} |\nabla \theta|^2 \, dx\, dt + \varepsilon \sup_{t \in (T,0)} \int_{B_R} |\theta(x,t)|^2 \, dx.
	\end{aligned}
\end{equation}
By Widman's hole-filling trick, there exists $\gamma := \max \{ (N+1)/(N+2), 2\varepsilon \} \in (0,1)$ satisfying
\begin{equation}
\begin{aligned}
	&\frac{1}{2(N+2)} \sup_{t \in (\tau,0)} \int_{B_\varrho} |\theta(x,t)|^2 \, dx + \iint_{Q_{\varrho,\tau}} |\nabla \theta|^2 \, dx\, dt \leq C(N) \times \mathbf{C} \iint_{Q_{R,T} \setminus Q_{\varrho,\tau} } |\theta|^2 \, dx\, dt \\
	&\quad\quad + \gamma \iint_{Q_{R,T}} |\nabla \theta|^2 \, dx\, dt + \frac{\gamma}{2(N+2)} \sup_{t \in (T,0)} \int_{B_R} |\theta(x,t)|^2 \, dx.
\end{aligned}
\end{equation}
To remove the extra terms on the RHS, we use a standard iteration argument on a sequence of scales (progressing `outward') $\varrho_0 = \varrho$, $\varrho_{k+1} = \varrho + (1-\lambda^{k+1}) (R-\varrho)$, $R_k =\varrho_{k+1}$, $\tau_0 = \tau$, $\tau_{k+1} = \tau + (1-\lambda^{2k+2}) (T-\tau)$, $T_k = \tau_{k+1}$, $k=0,1,2,\hdots$, where $0 < \lambda < 1$ is defined by the relation $\lambda^{\max(\alpha,2)} = \gamma/2$. See~\cite[p. 191, Lemma 6.1]{giustibook}, for example. This gives the desired Cacciopoli inequality.
\end{proof}

Next, we require a simple corollary:

\begin{corollary}[Interpolation inequality]
\label{cor:interpcor}
Let $\chi = 1+2/n$. Then
\begin{equation}
	\label{eq:interpinequality}
		\left( \iint_{Q_{\varrho,\tau}} |\theta|^{2\chi} \, dx\, dt \right)^{\frac{1}{\chi}} \les_{N,\alpha,\varepsilon}  \mathbf{C} \iint_{Q_{R,T} \setminus Q_{\varrho,\tau}} |\theta|^2 \, dx\, dt.
\end{equation}
\end{corollary}
\begin{proof}
Let $0 \leq \varphi \in C^\infty_0(B_{(R+\varrho)/2})$ satisfying $\varphi \equiv 1$ on $B_\varrho$ and $|\nabla \varphi| \les 1/(R-\varrho)$. Using~\eqref{eq:cacciopoliinequality} at an intermediate scale, we find
	\begin{equation}
	\label{eq:cacciopoliinequalitydos}
	\begin{aligned}
		\sup_{t \in (\tau,0)} \int_{B_\varrho} |\theta(x,t)\varphi|^2 \, dx + \iint_{Q_{\varrho,\tau}} |\nabla (\theta\varphi)|^2 \, dx\, dt \les_{N,\alpha,\varepsilon} \mathbf{C} \iint_{Q_{R,T} \setminus Q_{\varrho,\tau}} |\theta|^2 \, dx\, dt.
		\end{aligned}
\end{equation}
Then~\eqref{eq:interpinequality} follows from the Gagliardo-Nirenberg inequality on the whole space. 
\end{proof}


We are now ready to use Moser's iteration:

\begin{proof}[Proof of local boundedness]
Let $\beta_k := \chi^k$, where $k=0,1,2,\hdots$. A standard computation implies that $\theta^{\beta_k}$ is also a non-negative Lipschitz subsolution. Hence, it satisfies the Cacciopoli inequality~\eqref{eq:interpinequality} with $R_k = \varrho+2^{-k}(R-\varrho)$, $r_k = R_{k+1}$, $T_k = \tau-2^{-2k}(\tau-T)$, $\tau_k=T_{k+1}$, $k=0,1,2, \hdots$ (iterating inward). In other words,
\begin{equation}
	\label{eq:iterationineqforfk}
	\norm{\theta^{2\beta_{k+1}}}_{L^1(Q_{\varrho_k,\tau_k})}^{\frac{1}{\chi}} \les_{N,\alpha,\varepsilon} \mathbf{C}(\varrho_k,\tau_k,R_k,T_k,M,\delta,\alpha)
	\norm{\theta^{2\beta_k}}_{L^1(Q_{R_k,T_k} \setminus Q_{\varrho_k,\tau_k})}.
\end{equation}
We may expand the domain of integration on the RHS as necessary.
Define
\begin{equation}
	M_0 := \norm{\theta^2}_{L^1(Q_{R_0,T_0} \setminus Q_{\varrho_0,\tau_0})}
	\end{equation}
	and
	\begin{equation}
	M_{k+1} := \norm{\theta^2}_{L^{\beta_{k+1}}(Q_{\varrho_k,\tau_k})} = \norm{\theta^{2\beta_{k+1}}}_{L^1(Q_{\varrho_k,\tau_k})}^{\frac{1}{\beta_{k+1}}}, \quad k=0,1,2,\hdots
\end{equation}
Raising~\eqref{eq:iterationineqforfk} to $1/\beta_k$ and using Eq.~\eqref{eq:Cdef} defining $\mathbf{C}$, we obtain
\begin{equation}
	M_{k+1} \leq C(N,\alpha,\varepsilon)^{\frac{1}{\beta_k}} 2^{\frac{\max(2,\alpha) k}{\beta_k}} \mathbf{C}(\varrho,\tau,R,T,M,\delta,\alpha)^{\frac{1}{\beta_k}} M_k.
\end{equation}
Iterating, we have
\begin{equation}
	M_{k+1} \leq C(N,\alpha,\varepsilon)^{\sum_{j=0}^k \frac{1}{\chi^j}} 2^{\sum_{j=0}^k \frac{\max(2,\alpha) j}{\chi^j}} \mathbf{C}(\varrho,\tau,R,T,M,\delta,\alpha)^{\sum_{j=0}^k \frac{1}{\chi_j}} M_0.
\end{equation}
Finally, we send $k \to +\infty$ and substitute $\sum_{j \geq 0} 1/\chi^j = (n+2)/2$ to obtain
\begin{equation}
	\label{eq:intermediatelocalbddness}
	\norm{\theta}_{L^\infty(Q_{\varrho,\tau})} \les_{N,\alpha,\varepsilon} \mathbf{C}^{\frac{n+2}{4}} \norm{\theta}_{L^2(Q_{R,T} \setminus Q_{\varrho,\tau})}.
\end{equation}

We now demonstrate how to replace $L^2$ on the RHS of~\eqref{eq:interpinequality} with $L^\gamma$ ($0 < \gamma < 2$). To begin, use the interpolation inequality $\norm{\theta}_{L^2} \leq \norm{\theta}_{L^\gamma}^{\gamma/2} \norm{\theta}_{L^\infty}^{1-\gamma/2}$ in~\eqref{eq:intermediatelocalbddness} and split the product using Young's inequality. This gives
\begin{equation}
	\label{eq:loweringthenorm}
	\norm{\theta}_{L^\infty(Q_{\varrho,\tau})} \leq C(N,\alpha,\varepsilon,\gamma) \mathbf{C}^{\frac{n+2}{2\gamma}} \norm{\theta}_{L^\gamma(Q_{R,T} \setminus Q_{\varrho,\tau})} + \frac{1}{2} \norm{\theta}_{L^\infty(Q_{R,T})}.
\end{equation}
The second term on the RHS is removed by iterating outward along a sequence of scales, as in the proof of the Cacciopoli inequality in Lemma~\ref{lem:cacciopoli}.
\end{proof}

\begin{remark}[Elliptic case]
The analogous elliptic result is
\begin{equation}
	\sup_{B_\varrho} \theta \les_{N,\alpha,\varepsilon,\gamma} \mathbf{C}^{\frac{n}{2\gamma}} \norm{\theta}_{L^\gamma(B_R \setminus B_\varrho)},
	\end{equation}
	where $\mathbf{C}(\varrho,R,M,\delta,\alpha) = 1/[\delta^2 (R-\varrho)^2] + MR_0^{\alpha-2}/[\delta^\alpha (R-\varrho)^\alpha]$.
The proof is the same except that $\chi = n/(n-2)$ and $\sum 1/\chi^j = n/2$.
\end{remark}

\subsection{Proof of Harnack inequality}

In this subsection, $\theta$ is a strictly positive Lipschitz solution.\footnote{There is no loss of generality if we replace $\theta$ by $\theta + \kappa$ and let $\kappa \to 0^+$.} Then $\log \theta$ is well defined. Let $0 \leq \psi \in C^\infty_0(B)$ be a radially decreasing function satisfying $\psi \equiv 1$ on $B_{3/4}$. We use the notation
\begin{equation}
	f_\avg = \frac{1}{\Vol} \int_{\R^n} f \psi^2 \, dx, \quad \Vol = \int_{\R^n} \psi^2 \, dx,
\end{equation}
whenever $f \in L^1_\loc(B)$.
Let
\begin{equation}
\label{eq:Kdef}
	\mathbf{K} = \exp \left( \log \theta(\cdot,0) \right)_\avg.
\end{equation}
whose importance will be made clear in the proof of Lemma~\ref{lem:universal}.
Define
\begin{equation}
	v = \log \left( \frac{\theta}{\mathbf{K}} \right).
\end{equation}
 Then $v(\cdot,0)_\avg = 0$. A simple computation yields
\begin{equation}
	\label{eq:supsol}
	|\nabla v|^2 = \p_t v - \Delta v + b \cdot \nabla v.
\end{equation}
That is, $v$ is itself a supersolution, though it may not itself be positive. We crucially exploit that $|\nabla v|^2$ appears on the LHS of~\eqref{eq:supsol}. First, we require the following decomposition of the drift:

\begin{lemma}[Decomposition of drift]
\label{lem:decompositionofdrift}
We have the following decomposition on $B_{3/4} \times (-T^*,T^*)$:
\begin{equation}
	b = b_1 + b_2, \quad b_1 = - \div a, \quad \div b_2 = 0,
\end{equation}
where $a : B_{3/4} \times (-T^*,T^*) \to \R^{n\times n}_{\rm anti}$ is antisymmetric and
\begin{equation}
	\label{eq:decompest}
	\| a \|_{L^2(B_{3/4} \times (-T^*,T^*))} + \| b_2 \|_{L^2(B_{3/4} \times (-T^*,T^*))} \les \| b \|_{L^2_t H^{-1}_x(Q^*)},
\end{equation}
where $Q^* = B_1 \times (-T^*,T^*)$.
\end{lemma}
\begin{proof}
Let $\phi \in C^\infty_0(B_1)$ with $\phi \equiv 1$ on $B_{15/16}$. Let $\tilde{b}=\phi b$. Hence, $\| \tilde{b} \|_{L^2_t H^{-1}_x(Q^*)} \les \| b \|_{L^2_t H^{-1}_x(Q^*)}$. We may decompose $\tilde{b}(\cdot,t) \in H^{-1}(\R^n)$ into $\tilde{b}_1(\cdot,t) \in \dot H^{-1}(\R^n)$, whose Fourier transform is supported outside of $B_2$, and $\tilde{b}_2(\cdot,t) \in L^2(\R^n)$. Define 
\begin{equation}
	a_{ij} = \Delta^{-1} ( - \p_j \tilde{b}_{1i} + \p_i \tilde{b}_{1j} ), \quad g = \Delta^{-1} ( - \div \tilde{b}_1 ).
\end{equation}
This amounts to performing the Hodge decomposition in $\R^n$ `by hand'.\footnote{We are simply exploiting the identity $\Delta = dd^* + d^*d$ on differential $k$-forms, up to a sign convention, for differential $1$-forms $\cong$ vector fields.} Clearly, $a$ is antisymmetric, and we have the decomposition
\begin{equation}
	- \tilde{b}_1 = \div a + \nabla g
\end{equation}
and the estimates
\begin{equation}
	\| a(\cdot,t) \|_{L^2(\R^n)} + \| g(\cdot,t) \|_{L^2(\R^n)} \les \| \tilde{b}_1(\cdot,t) \|_{\dot H^{-1}(\R^n)}.
\end{equation}
Similarly, we decompose
\begin{equation}
	\tilde{b}_2 = \mathbb{P} \tilde{b}_2 + \mathbb{Q} \tilde{b}_2,
\end{equation}
where $\mathbb{P}$ is the Leray (orthogonal) projector onto divergence-free fields, and $\mathbb{Q} = I - \mathbb{P}$ is the orthogonal projector onto gradient fields. We denote $\mathbb{Q} \tilde{b}_2 = \nabla f$.

Since $\tilde{b} = \tilde{b}_1 + \tilde{b}_2 = - \div a + \mathbb{P} \tilde{b}_2 + \nabla (f - g)$ is divergence free in $B_{7/8}$ on time slices, we have
\begin{equation}
	\Delta (f- g)(\cdot,t) = 0 \text{ in } B_{7/8},
\end{equation}
and by elliptic regularity, for all $k \geq 0$,
\begin{equation}
\| \nabla (f - g)(\cdot,t) \|_{H^k(B_{3/4})} \les_k \| \tilde{b}_1(\cdot,t) \|_{\dot H^{-1}(B_{7/8})} + \| \tilde{b}_2(\cdot,t) \|_{L^2(B_{7/8})}.
\end{equation}
Finally, we define
\begin{equation}
	b_2 = \mathbb{P} \tilde{b}_2 + \nabla (f - g) \in L^2_{t,x}(B_{3/4} \times (-T^*,T^*)),
\end{equation}
which satisfies the claimed estimates and is divergence free in $B_{7/8} \times (-T^*,T^*)$.
\end{proof}

We now proceed with the proof of Harnack's inequality:

\begin{lemma}
\label{lem:universal}
For all non-zero $t \in [-T^*,T^*]$, we write $I_t = [0,t]$ if $t > 0$ and $I_t = [t,0]$ if $t < 0$. Then
\begin{equation}
	\label{eq:universalest}
	-\sgn(t) v(\cdot,t)_\avg + \int_{I_t} \left(|\nabla v|^2(\cdot,s) \right)_\avg \, ds \les  |t| + \norm{b}_{L^2_t H^{-1}_x(B \times I_t)}^2.
\end{equation}
\end{lemma}
\begin{proof}
We multiply~\eqref{eq:supsol} by $\psi^2$ and integrate over $B \times I_t$:
\begin{equation}
\label{eq:harnack1}
	\sgn(t) \int_B \left( v(x,0) - v(x,t) \right) \psi^2   \, dx +  \iint_{B \times I_t} |\nabla v|^2 \psi^2 \, dx\, ds \leq  \iint_{B \times I_t} 2 \psi \nabla \psi \cdot \nabla v + (b \cdot \nabla v) \psi^2 \, dx\, ds.
\end{equation}
By~\eqref{eq:Kdef}, $\int_{\R^n} v(x,0) \psi^2 \, dx = 0$.
The first term on the RHS is easily estimated:
\begin{equation}
	 \iint_{B \times I_t} 2 \psi \nabla \psi \cdot \nabla v \, dx\, ds \leq \frac{1}{4} \iint_{B \times I_t} |\nabla v|^2 \psi^2 \, dx\, ds + C|t|.
\end{equation}
To estimate the term containing $b$, we require the drift decomposition $b = b_1 + b_2$ from Lemma~\ref{lem:decompositionofdrift}. Then
\begin{equation}
	\label{eq:harnack3}
	\iint_{B \times I_t} ( b_1 \cdot \nabla v) \psi^2 \, dx \,ds = \iint_{B \times I_t} 2\psi a(\nabla \psi,\nabla v) \, dx\, ds \leq  \frac{1}{4} \iint_{B \times I_t} |\nabla v|^2 \psi^2 \, dx\, ds + C \norm{a}_{L^2(B \times I_t)}^2.
\end{equation}
and
\begin{equation}
	\label{eq:harnack4}
	\iint_{B \times I_t} ( b_2 \cdot \nabla v) \psi^2 \, dx \,ds \leq  \frac{1}{4} \iint_{B \times I_t} |\nabla v|^2 \psi^2 \, dx\, ds + C \norm{b_2}_{L^2(B \times I_t)}^2.
\end{equation}
Recall the estimate~\eqref{eq:decompest} from the decomposition. Combining~\eqref{eq:harnack1}--\eqref{eq:harnack4} and dividing by $\Vol$ gives~\eqref{eq:universalest}.
\end{proof}

In the following, we write $v = v_+ - v_-$, where $v_+, v_- \geq 0$. We also use the notation
\begin{equation}
	A^+ = \int_0^{T^*} \| b(\cdot,t) \|_{H^{-1}(B)}^2 \, dt, \quad A^- = \int_{-T^*}^0 \| b(\cdot,t) \|_{H^{-1}(B)}^2 \, dt.
\end{equation}

\begin{lemma}[Weak--$L^1$ estimates]
\label{lem:weakl1est}
With the above notation, we have
\begin{equation}
\norm{v_+}_{L^{1,\infty}_{t,x}(B_{3/4} \times (-T^*,0))} \les 1+ T^*(T^* + A^-)
\end{equation}
and
\begin{equation}
	\norm{v_-}_{L^{1,\infty}_{t,x}(B_{3/4} \times (0,T^*))} \les 1+T^*(T^* + A^+).
\end{equation}
\end{lemma}
\begin{proof}
By~\eqref{eq:universalest} and a weighted Poincar{\'e} inequality \cite[Lemma 3, p. 120]{moserharnack},
\begin{equation}
	\label{eq:wewishtorealize}
	-\sgn(t) v_\avg(t) + \frac{1}{C_1} \int_{I_t} \left( |v-v_\avg|^2 \right)_\avg \, ds \leq C_0 \int_{I_t} \left( 1 + \| b(\cdot,t) \|_{H^{-1}(B)}^2 \right) ds,
\end{equation}
 where $C_0 > 0$ is the implied constant in~\eqref{eq:universalest}.
In the following, we focus on the case $t \in [-T^*,0]$.
We use~\eqref{eq:wewishtorealize} to obtain a sub/supersolution inequality corresponding to a quadratic ODE. First, we remove the forcing in the ODE by defining
\begin{equation}
	\label{eq:pdef}
	p(x,t) := v(x,t) - C_0 \underbrace{\int_{I_t} \left( 1 + \| b(\cdot,t) \|_{H^{-1}(B)}^2 \right) ds}_{\leq T^* + A^-}.
\end{equation}
Then~\eqref{eq:wewishtorealize} becomes
\begin{equation}
	p_\avg(t) + \frac{1}{C_1} \int_{I_t} \left( |p-p_\avg|^2 \right)_\avg \, ds \leq 0.
\end{equation}
Let us introduce the super-level sets, whose measures $\eta$ appear as a coefficient in the ODE:
\begin{equation}
	\label{eq:etadef}
	\eta(\mu,t) := \left|\{ x\in B_{3/4} : p(x,t) > \mu\}\right|, \quad \mu > 0.
\end{equation}
Since $p_\avg \leq 0$, we have that $p(x,t) - p_\avg(t) > \mu - p_\avg(t) > 0$ whenever $p(x,t) > \mu$. Then
\begin{equation}
	\label{eq:rephrase}
	p_\avg(t) + \frac{1}{C_1 \Vol} \int_{I_t} \eta(\mu,s) (\mu-p_\avg)^2 \, ds \leq 0.
\end{equation}
It is convenient to rephrase~\eqref{eq:rephrase} in terms of a positive function evolving forward-in-time: $\bar{p}(t) = -p_\avg(-t)$ with $t \in [0,T^*]$. Then~\eqref{eq:rephrase} becomes
\begin{equation}
	\bar{p}(t) \geq \frac{1}{C_1 \Vol} \int_0^t \eta(\mu,|s|) (\mu+\bar{p}(s))^2 \, ds.
\end{equation}
The above inequality means that $\bar{p}$ is a supersolution of the quadratic ODE
\begin{equation}
	\label{eq:qode}
	\dot q = \frac{1}{C_1 \Vol} \times \eta(\mu,|t|) (\mu+q)^2
\end{equation}
with $q(0) = 0$. The above scalar ODE has a comparison principle. \emph{A priori}, since~\eqref{eq:qode} is quadratic, its solutions may quickly blow-up depending on the size of $\eta(\mu,\cdot)$ and $\mu$. However, because $\bar{p}$ lies above the solution $q$, $q$ does not blow up, and we obtain a bound for the density $\eta(\mu,\cdot)$ in the following way. After separating variables in~\eqref{eq:qode}, we obtain
\begin{equation}
	\frac{1}{C_1 \Vol} \int_0^{T^*} \eta(\mu,|s|) \, ds = \frac{1}{\mu} - \frac{1}{\mu+q(T^*)} \leq \frac{1}{\mu},
\end{equation}
since $q \geq 0$. That is,
\begin{equation}
	\norm{p_+}_{L^{1,\infty}(B_{3/4} \times (-T^*,0))} \les 1.
\end{equation}
Finally, since $\norm{\cdot}_{L^{1,\infty}(B_{3/4} \times (-T^*,0))}$ is a quasi-norm and $v \leq p + C_0 (T^* + A^-)$ pointwise due to \eqref{eq:pdef}, we have
\begin{equation}
	\begin{aligned}
	\norm{v_+}_{L^{1,\infty}(B_{3/4} \times (-T^*,0))} &\leq 2\norm{p_+}_{L^{1,\infty}(B_{3/4} \times (-T^*,0))} + 2 C_0 \norm{T^* + A^-}_{L^{1,\infty}(B_{3/4} \times (-T^*,0))} \\ &\les 1 +  T^*(T^* + A^-).
	\end{aligned}
\end{equation}
The proof for $t \in [0,T^*]$ is similar except that one uses sub-level sets in~\eqref{eq:etadef} with $\mu < 0$.
\end{proof}

We now require the following lemma of Moser~\cite{moserpointwise}, which we quote almost directly, and in which we denote by $Q(\varrho)$, $\varrho > 0$ any family of domains satisfying $Q(\varrho) \subset Q(r)$ for $0 < \varrho < r$.
\begin{lemma}[Lemma~3 in~\cite{moserpointwise}]
\label{lem:moserlem}
Let $m, \zeta, c_0$, $1/2 \leq \theta_0 < 1$ be positive constants, and let $w > 0$ be a continuous function defined in a neighborhood of $Q(1)$ for which
\begin{equation}
	\label{eq:moserlocalboundednessreq}
	\sup_{Q(\varrho)} w^p < \frac{c_0}{(r-\varrho)^m {\rm meas}(Q(1))} \iint_{Q(r)} w^p \, dt \, dx
\end{equation}
for all $\varrho, r, p$ satisfying
\begin{equation}
	\frac{1}{2} \leq \theta_0 \leq \varrho < r \leq 1, \quad 0 < p < \zeta^{-1}.
\end{equation}
Moreover, let
\begin{equation}
	\label{eq:moserweakl1req}
	{\rm meas} \{ (x,t) \in Q(1) : \log w > \mu \} < \frac{c_0 \zeta}{s} {\rm meas}(Q(1))
\end{equation}
for all $\mu > 0$. Then there exists a constant function $\gamma = \gamma(\theta_0,m,c_0)$ such that
\begin{equation}
	\label{eq:moserlemconclusion}
	\sup_{Q(\theta_0)} w < \gamma^\zeta.
\end{equation}
\end{lemma}

\begin{proof}[Proof of Harnack inequality]
We apply Lemma~\ref{lem:moserlem} to $w = \theta / \mathbf{K}$ with $Q(\varrho) = B_\varrho \times (-T^* + 2\ell (1-\varrho),0)$ and $\theta_0 = 1/2$.\footnote{Technically, to satisfy the conditions in Lemma~\ref{lem:moserlem}, $w = \theta / \mathbf{K}$ should be extended arbitrarily to be continuous in a neighborhood of $Q(1)$.} Indeed, we recognize the requirement~\eqref{eq:moserlocalboundednessreq} as the local boundedness guaranteed by Theorem~\ref{thm:localboundedness} and~\eqref{eq:moserweakl1req} as the weak $L^1$ estimate from Lemma~\ref{lem:weakl1est}. This gives
\begin{equation}
	\label{eq:estimateone}
\sup_{B_{1/2} \times (-T^* + \ell,0)} \frac{\theta}{\mathbf{K}} \les 1.
\end{equation}
Here, we suppress also the dependence on the time lag $\ell$.
Meanwhile, $v_- = \log_+ (\mathbf{K}/\theta)$ is a subsolution. Hence,
\begin{equation}
	\norm{v_-}_{L^\infty(B_{1/2} \times (\ell,T^*))} \les \norm{v_-}_{L^{1,\infty}(B_{3/4} \times (0,T^*))}.
\end{equation}
On the other hand,
\begin{equation}
	\label{eq:estimatetwo}
	\frac{\mathbf{K}}{\inf \theta} = \sup \frac{\mathbf{K}}{\theta} = \exp \left( \sup \log \frac{\mathbf{K}}{\theta} \right) \leq \exp \left( \sup v_- \right) \les \exp \left( \norm{v_-}_{L^{1,\infty}} \right) \les 1,
\end{equation}
where the $\inf$ and $\sup$ are taken on $B_{1/2} \times (\ell,T^*)$.
Combining~\eqref{eq:estimateone} and~\eqref{eq:estimatetwo}, we arrive at
\begin{equation}
	\sup_{B_{1/2} \times (-T^*+\ell,0)} \theta \les \mathbf{K} \les \inf_{B_{1/2} \times (\ell,T^*)} \theta,
\end{equation}
as desired.
\end{proof}

\section{Bounded total speed}
In this section, we prove the statements in Theorem~\ref{thm:lqlp} concerning the space $L^1_t L^\infty_x$.

\begin{proposition}[Local boundedness]
\label{pro:localbddbddtotal}
Let $T \in (-\infty,0)$  and $\tau \in (T,0)$. Let $b \: Q_{1,T} \to \R^n$ be a smooth divergence-free drift satisfying
\begin{equation}
                            \label{eq1.40}
	\norm{b}_{L^1_t L^\infty_x(Q_{1,T})} \leq 1/8.
\end{equation}
Let $\theta$ be a non-negative Lipschitz subsolution on $Q_{1,T}$.
Then, for all $\gamma \in (0,2]$, we have
\begin{equation}
	\label{eq:localbddnessforbddtotalspeed}
	\| \theta \|_{L^\infty(Q_{1/2,\tau})} \les_{\gamma} \left(1 + \frac{1}{\tau-T} \right)^{\frac{n+2}{2\gamma}} \norm{\theta}_{L^\gamma(Q_{1,T})}.
\end{equation}
\end{proposition}

\begin{proof}
For smooth $\lambda : [T,0] \to (0,+\infty)$, we define $x = \lambda y$ and
\begin{equation}
	\tilde{\theta}(y,t) =  \theta(\lambda y,t).
\end{equation}
That is, $\tilde{\theta}$ is obtained by dynamically rescaling $\theta$ in space. The new PDE is
\begin{equation}
	\label{eq:resultingeqnis}
	\p_t \tilde{\theta} - \frac{1}{\lambda^2} \Delta_y \tilde{\theta} + \frac{1}{\lambda} \tilde{b} \cdot \nabla_y \tilde{\theta} \leq 0
\end{equation}
where
\begin{equation}
	\label{eq:tildebdef}
	\tilde{b}(y,t) = b(\lambda y,t) - \dot \lambda y.
\end{equation}
Choose $\lambda(T) = 1$, $\dot \lambda = - 2\| b(\cdot,t) \|_{L^\infty}$ when $t \in [T,0]$. Clearly, $3/4 \leq \lambda \leq 1$. Our picture is that $\tilde{\theta}$ dynamically `zooms in' on $\theta$.
In particular, using \eqref{eq1.40} and~\eqref{eq:tildebdef},
\begin{equation}
	\label{eq:tildebpointsoutward}
	\tilde{b}(\cdot,t) \cdot \frac{y}{|y|} \geq - \| b(\cdot,t) \|_{L^\infty} + 2  \| b(\cdot,t) \|_{L^\infty}|y| \geq 0 \text{ when } y \in B_{1} \setminus B_{1/2},
	\end{equation}
and
\begin{equation}
	\div \tilde{b} =  2n \| b(\cdot,t) \|_{L^\infty} \geq 0.
\end{equation}
We now demonstrate Cacciopoli's inequality in the new variables. Let $3/4 \leq \varrho < R \leq 1$. Let $\varphi \in C^\infty_0(B_{R})$ be a radially symmetric and decreasing function satisfying  $0 \leq \varphi \leq 1$ on $\R^n$, $\varphi \equiv 1$ on $B_\varrho$, and $|\nabla \varphi| \les 1/(R - \varrho)$. Let $\eta \in C^\infty_0(T,+\infty)$ satisfying $0 \leq \eta \leq 1$ on $\R$, $\eta \equiv 1$ on $(\tau,0)$, and $|d\eta/dt| \les 1/(\tau-T)$.
Let $\Phi = \varphi^2 \eta$. We integrate Eq.~\eqref{eq:resultingeqnis} against $\tilde{\theta} \Phi$ on $B_{R} \times (T,t)$ for $t\in (\tau,0)$. Then
\begin{equation}
\begin{aligned}
	&\frac{1}{2} \int  |\tilde{\theta}|^2(y,t) \Phi \, dy + \iint \lambda^{-2} |\nabla \tilde{\theta}|^2 \Phi \, dy\, ds \\
	&\quad \leq \frac{1}{2} \iint  |\tilde{\theta}|^2 \p_t \Phi \, dy\, ds
	- \iint \lambda^{-2} \tilde{\theta} \nabla \tilde{\theta} \cdot \nabla \Phi \, dy\, ds  \\
	&\quad\quad + \frac{1}{2} \iint \lambda^{-1} \underbrace{\tilde{b} \cdot \nabla \Phi}_{\leq 0 \text{ by }\eqref{eq:tildebpointsoutward}} |\tilde{\theta}|^2 \, dy\, ds + \frac{1}{2} \iint \lambda^{-1} \div \tilde{b}  |\tilde{\theta}|^2 \Phi \, dy\, ds.
	\end{aligned}
\end{equation}
While $\div \tilde{b}$ has a disadvantageous sign, it simply acts as a bounded potential.
Simple manipulations give the Cacciopoli inequality:
\begin{equation}
\begin{aligned}
	 &\sup_{t \in (\tau,0)} \int_{B_\varrho} |\tilde{\theta}|^2(y,t) \, dy + \iint_{Q_{\varrho,\tau}}  |\nabla \tilde{\theta}|^2 \varphi^2 \, dy\, ds  \\
	 &\quad \les \left( \frac{1}{\tau-T} + \frac{1}{(R-\varrho)^2} \right) \iint_{Q_{R,T} \setminus Q_{\varrho,\tau}} |\tilde{\theta}|^2 \, dy\, ds.
	 \end{aligned}
\end{equation}
The remainder of the proof proceeds as in Theorem~\ref{thm:localboundedness} except in the $(y,t)$ variables. Namely, we have the interpolation inequality as in Corollary~\ref{cor:interpcor}, and $\tilde{\theta}^\beta$ is a subsolution of~\eqref{eq:resultingeqnis} whenever $\beta \geq 1$. Therefore, we may perform Moser's iteration verbatim. As in~\eqref{eq:loweringthenorm}, the $L^2$ norm on the RHS may be replaced by the $L^\gamma$ norm. Finally, undoing the transformation yields the inequality~\eqref{eq:localbddnessforbddtotalspeed} in the $(x,t)$ variables, since $(y,t) \in B_R \times \{ t \}$ corresponds to $(x,t) \in B_{\lambda(t)R} \times \{ t \}$.
\end{proof}

The quantitative local boundedness property in Theorem~\ref{thm:lqlp} follows from applying Proposition~\ref{pro:localbddbddtotal} and its rescalings on finitely many small time intervals. In Remark~\ref{rmk:boundedtotalspeedassertion}, we justify that the constant depends on the `profile' of $b$ and not just its norm. 

\section{Counterexamples}
\label{sec:counters}

\subsection{Elliptic counterexamples}
Let $n \geq 3$. Our counterexamples will be axisymmetric in `slab' domains $B_R \times (0,1)$, where $R>0$ is arbitrary and $B_R$ is a ball in $\R^{n-1}$. We use the notation $x = (x',z)$, where $x' \in \R^{n-1}$, $r = |x'|$, and $z \in (0,1)$. Let
\begin{equation}
\ubar{\theta}(x) = u(r) z,
\end{equation}
\begin{equation}
b(x) = V(r) e_z.
\end{equation}
Since $b$ is a shear flow in the $e_z$ direction, it is divergence free. Then
\begin{equation}
	\label{eq:incylindrical}
	- \Delta \ubar{\theta} + b \cdot \nabla \ubar{\theta} = - z \Delta_{x'} u + V u.
\end{equation}
We will a construct a subsolution $\ubar{\theta}$ and supersolution $\bar{\theta}$ using the \emph{steady Schr{\"o}dinger equation}
\begin{equation}
	\label{eq:steadyschrodinger}
	-\Delta u + V u = 0
\end{equation}
in dimension $n-1$, where additionally $u \geq 0$ and $V \leq 0$. By the scale invariance in $x'$, it will suffice to construct a solution at a single fixed length scale $R = R_0$. The way to proceed is well known. We define
\begin{equation}
	u = \log \log \frac{1}{r},
\end{equation}
\begin{equation}
	V = \frac{\Delta u}{u},
\end{equation}
for $r \leq R_0 \ll 1$ so that $u$ is well defined.
A simple calculation verifies that $0 \leq u \in H^1_\loc$, $\Delta u, V \le 0$, and $\Delta u, V \in L^{(n-1)/2}_\loc$.\footnote{Since Schr{\"o}dinger solutions with critical potentials $V$ belong to $L^p_\loc$ for all $p < \infty$ (see Han and Lin~\cite{hanlinbook}, Theorem 4.4), it is natural to choose $u$ with a $\log$. The double $\log$ ensures that $u$ has finite energy when $n=3$. Notice also that $\Delta u = - (r \log r^{-1})^{-2}$ when $n = 3$.} Therefore, $Vu = \Delta u \in L^1_\loc$, and the PDE~\eqref{eq:steadyschrodinger} is satisfied in the sense of distributions. Using~\eqref{eq:incylindrical}, we verify that $\ubar{\theta}$ is a distributional subsolution:
\begin{equation}
	- \Delta \ubar{\theta} + b \cdot \nabla \ubar{\theta} = - z \Delta_{x'} u + V u = (1-z) Vu \leq 0 \text{ in } B_{R_0} \times (0,1),
\end{equation}
with equality at $\{ z = 1 \}$. We also wish to control solutions from above. 
Since $\Delta u \leq 0$, we define
\begin{equation}
	\bar{\theta}(x',z) = u(r).
\end{equation}
Clearly, $\ubar{\theta} \leq \bar{\theta}$, and $\bar{\theta}$ is a distributional supersolution:
\begin{equation}
	-\Delta \bar{\theta} + b \cdot \nabla \bar{\theta} = - \Delta_{x'} u \geq 0 \text{ in } B_{R_0} \times (0,1).
\end{equation}

We now construct smooth subsolutions and supersolutions approximating $\ubar{\theta}$ and $\bar{\theta}$ according to the above procedure. Let $\varphi$ be standard mollifier and
\begin{equation}
\varphi_\varepsilon = \frac{1}{\varepsilon^{n-1}} \varphi \left( \frac{\cdot}{\varepsilon} \right).
\end{equation}
 Define $u_\varepsilon = \varphi_\varepsilon \ast u$, $V_\varepsilon = \Delta u_\varepsilon/u_\varepsilon$, $b_\varepsilon = V_\varepsilon(r) e_z$, $\ubar{\theta}_{\varepsilon} = z u_\varepsilon(r)$, and $\bar{\theta}_\varepsilon = u_\varepsilon(r)$. Then $(\ubar{\theta}_\varepsilon)$ and $(\bar{\theta}_\varepsilon)$ trap a family $(\theta_\varepsilon)$ of smooth solutions to the PDEs
 \begin{equation}
	- \Delta \theta_\varepsilon + b_\varepsilon \cdot \nabla \theta_\varepsilon = 0 \text{ on } B_{R_0/2} \times (0,1)
 \end{equation}
 when $\varepsilon \in (0,R_0/2)$. Moreover, we have the desired estimates
 \begin{equation}
	\sup_{\varepsilon \in (0,R_0/2)} \| \theta_\varepsilon \|_{L^p(B_{R_0/2} \times (0,1))} \leq \| \bar{\theta} \|_{L^p(B_{R_0} \times (0,1))} < +\infty, \quad p \in [1,+\infty),
 \end{equation}
 \begin{equation}
	\sup_{\varepsilon \in (0,R_0/2)} \| V_\varepsilon \|_{L^{\frac{n-1}{2}}(B_{R_0/2})} \les \sup_\varepsilon \| \Delta u_\varepsilon \|_{L^{\frac{n-1}{2}}(B_{R_0/2})} \les \| \Delta u \|_{L^{\frac{n-1}{2}}(B_{R_0})} < +\infty,
 \end{equation}
 and the singularity as $\varepsilon \to 0^+$:
\begin{equation}
	\sup_{B_{\frac{R_0}{4}} \times (\frac{1}{4},\frac{3}{4})} \theta_\varepsilon \geq \sup_{B_{\frac{R_0}{4}} \times (\frac{1}{4},\frac{3}{4})} \ubar{\theta}_\varepsilon \to +\infty \text{ as } \varepsilon \to 0^+.
\end{equation}

\begin{remark}[Line singularity]
The solutions constructed above are singular on the $z$-axis, as the maximum principle demands.
\end{remark}

\begin{remark}[Time-dependent examples]
\label{rmk:timedependentexamples}
The above analysis of unbounded solutions for the steady Schr{\"o}dinger equation with critical potential is readily adapted to the parabolic PDE $\p_t u - \Delta u + V u = f$ in $B_{R_0} \times (-T_0,0) \subset \R^{n+1}$, $n \geq 2$, (i) with potential $V$ belonging to $L^q_t L^p_x$, $2/q + n/p = 2$, $q > 1$, and zero force, or (ii) with force $f$ belonging to the same space and zero potential. For example, one can define $u = \log \log (-t + r^2)$, $V = - (\p_t u - \Delta u)/u$, and $f=0$. The case $q=1$ is an endpoint case in which solutions remain bounded. These examples are presumably well known, although we do not know a suitable reference.
\end{remark}

\subsection{Parabolic counterexamples}

\begin{proof}[Proof of borderline cases: $L^q_t L^p_x$, $\frac{2}{q} + \frac{n}{p} = 2$, $q > 1$]

\textit{1. A heat subsolution}. Let
\begin{equation}
	\Gamma(x,t) = (4\pi t)^{-n/2} e^{-\frac{|x|^2}{4t}}
\end{equation}
be the heat kernel. Let
\begin{equation}
	E(x,t) = (\Gamma - c_n)_+,
\end{equation}
where $c_n = (8\pi)^{-n/2}$. Then $E$ is globally Lipschitz away from $t=0$, and $E(\cdot,t)$ is supported in the ball $B_{R(t)}$, where
\begin{equation}
	\label{eq:radiusdef}
	R(t)^2 = 2nt \log \frac{2}{t}, \quad t  < 2,
\end{equation}
and $E$ vanishes in $t \geq 2$. 

\textit{2. A steady, compactly supported drift}. There exists a divergence-free vector field $U \in C^\infty_0(B_{4})$ satisfying
\begin{equation}
	U \equiv \vec{e}_1 \text{ when } |x| \leq 2.
\end{equation}
Here is a construction: Let $\phi \in C^\infty_0(B_{4})$ be a radially symmetric cut-off function such that $\phi \equiv 1$ on $B_3$. By applying Bogovskii's operator in the annulus $B_{4} \setminus \overline{B_{2}}$, see~\cite[Theorem III.3.3, p.~179]{galdi}, there exists $W \in C^\infty_0(B_{4} \setminus \overline{B_{2}})$ solving
\begin{equation}
	\div W = - \div (\phi \vec{e}_1) \in C^\infty_0(B_{4} \setminus \overline{B_{2}}).
\end{equation}
Notably, the property of compact support is preserved. Finally, we define
\begin{equation}
	U = \phi \vec{e}_1 + W.
\end{equation}

\textit{3. Building blocks}. Let $0 \leq S \in C^\infty_0(0,1)$ and $X \: \R \to \R^n$ be the solution of the ODE
\begin{equation}
	\dot X(t) = S(t) \vec{e}_1, \quad X(0) = -10n \vec{e}_1.
\end{equation}
Define
\begin{equation}
	b_{S}(x,t) = S(t) U \left( \frac{x - X(t)}{R(t)} \right),
\end{equation}
where $R(t)$ was defined in~\eqref{eq:radiusdef} above,
and
\begin{equation}
	E_{S}(x,t) = E(x-X(t), t).
\end{equation}
Then $E_{S}$ is a subsolution:
\begin{equation}
	\left( \p_t - \Delta + b_{S} \cdot \nabla \right) E_{S} \leq 0 \quad \text{ on } \R^n \times (0,1).
\end{equation}
If $[a,a'] \subset (0,1)$, $S \in C^\infty_0(a,a')$, and $\int S \, dt \geq 20n$, then we have $E_{S}(\cdot,t)|_{B_3} \equiv 0$ when $t \leq a$ or $t \geq a'$. Additionally, $E_{S}(\cdot,\tilde{t}) = E(\cdot,\tilde{t})$ for some $\tilde{t} \in (a,a')$.

We also consider the solution $\Phi_S$ to the PDE:
\begin{equation}
\begin{aligned}
	\left( \p_t - \Delta + b_{S} \cdot \nabla \right) \Phi_S &= 0\quad  \text{ on } \R^{n+1} \setminus \{ (X(0),0) \} \\
	\Phi_S|_{t = 0} &= \delta_{x = X(0)},
	\end{aligned}
\end{equation}
which for short times $|t| \ll 1$ and negative times is equal to the heat kernel $\Gamma(x-X(0),t)$. By the comparison principle,
\begin{equation}
	E_{S} \leq \Phi_S.
\end{equation}

We have the following measurements on the size of the drift:
\begin{equation}
	\| b_S \|_{L^q_t L^p_x(\R^{n+1})}^q \les \int_{\R} S(t)^q R(t)^{\frac{nq}{p}} \, dt,
\end{equation}
where $1 \leq p,q < +\infty$,
and
\begin{equation}
	\| b_S \|_{L^1_t L^\infty_x(\R^{n+1})} = \|U\|_{L^\infty(\R^n)} \int_{\R} S(t) \, dt.
\end{equation}

\textit{4. Large displacement}. For $S_k \in C^\infty_0(t_k,t_k')$, $k \geq 1$, with $t_k = o_{k \to +\infty}(1)$ and $[t_k,t_k'] \subset (0,1)$ disjoint-in-$k$, we consider the drifts $b_{S_k}$. 
Let $M > 0$. We claim that it is possible to choose $S_k$ satisfying
\begin{equation}
	\norm{b_{S_k}}_{L^1_t L^\infty_x(\R^{n+1})} = M
\end{equation}
and
\begin{equation}
	\left \| \sum_{k \geq 1} b_{S_k} \right \|_{L^q_t L^p_x(\R^{n+1})} < +\infty
\end{equation}
for all $p,q \in [1,+\infty]$ satisfying $2/q+n/p = 2$ and $q>1$. Indeed, consider
\begin{equation}
                \label{eq7.12}
	\bar{S}(t) = \left( t \log t^{-1} \log \log t^{-1} \right)^{-1}
\end{equation}
when $t \leq c_0$ so that the above expression is well defined, and extended smoothly on $[c_0,1]$. We ask also that $t_1 \leq c_0$.
Since
\begin{equation}
	\label{eq:theintegral}
	\int_{a}^{a'} \bar{S}(t) dt = \log \log \log t^{-1} \big|_{a}^{a'}
\end{equation}
when $a' \leq c_0$, we have
\begin{equation}
	\int_{t=0}^1 \bar{S}(t) \, dt = +\infty,
\end{equation}
whereas
\begin{equation}
\begin{aligned}
	\int_{t=0}^1 \bar{S}(t)^q R(t)^{\frac{nq}{p}} \, dt &\leq
	O(1) + C_n \int_{t=0}^{c_0} (t \log t^{-1})^{-q+\frac{nq}{2p}} ( \log \log t^{-1} )^{-q} \, dt \\
	&\leq O(1) + C_n \int_{t=0}^{c_0} (t \log t^{-1})^{-1} ( \log \log t^{-1} )^{-q} \, dt < +\infty
	\end{aligned}
\end{equation}
when $q \in (1,+\infty)$. The case $q = +\infty$ is similar. We choose $S_k = \bar{S}(t) \varphi_k$ with suitable smooth cut-offs $\varphi_k$ to complete the proof of the claim.

\textit{5. Unbounded solution}. We choose $M = 20n$ and a suitable sequence of $S_k$ as above. We reorder the building blocks we defined above so that the $k$th subsolution and $k$th drift are `activated' on times $(1-t_k',1-t_k)$. Define
\begin{equation}
	b_k(\cdot, t) = b_{S_k}(\cdot,t - (1-t_k') + t_k), \quad b = \sum_{k \geq 1} b_k
\end{equation}
and, for size parameters $A_k \geq 0$,
\begin{equation}
	E_k(\cdot,t ) = E_{S_k}(\cdot, t- (1-t_k') + t_k) \mathbf{1}_{(1-t_k',1-t_k)}, \quad E = \sum_{k \geq 1} A_k E_k.
\end{equation}
Then $E$ is a subsolution of the PDE
\begin{equation}
	(\p_t - \Delta + b \cdot \nabla) E \leq 0\quad \text{ on } B_3 \times (-\infty,1).
\end{equation}
We further define
\begin{equation}
	\Phi_k(\cdot,t) = \Phi_{S_k}(\cdot, t- (1-t_k') + t_k), \quad \theta = \sum_{k \geq 1} A_k \Phi_k,
\end{equation}
which is a solution of the PDE
\begin{equation}
	(\p_t - \Delta + b \cdot \nabla) \theta = 0\quad \text{ on } B_2 \times (-\infty,1).
\end{equation}
Since $E_k \leq \Phi_k$, we have that $E \leq \theta$ on $B_2 \times (-\infty,1)$.
Additionally, we have
\begin{equation}
	\label{eq:lowerboundonslices}
	\sup_t A_k \| E_{k}(\cdot,t) \|_{L^\infty(B_1)} \geq  A_k \| E_{k}(\cdot,1-\tilde t_k) \|_{L^\infty(B_1)}  \ges A_k {t_k'}^{-n/2},
\end{equation}
where $\tilde t_k\in (t_k,t_k')$ satisfies $X_{S_k}(\tilde{t_k})=0$.
Therefore, by the comparison principle and~\eqref{eq:lowerboundonslices}, we have
\begin{equation}
	\label{eq:limsup}
	\limsup_{t \to 1_-} \| \theta(\cdot,t) \|_{L^\infty(B_1)} \ges \limsup_{k \to +\infty} A_k {t_k'}^{-n/2}.
\end{equation}
To control the solution from above, we use
\begin{equation}
	\label{eq:thingtoprune}
	\| \theta \|_{L^\infty_t L^1_x(\R^{n+1})} \leq \sum A_k.
\end{equation}
Therefore, it is possible to choose $A_k \to 0$ as $k \to +\infty$ while keeping the $\limsup$ in~\eqref{eq:limsup} infinite. Hence, by `pruning' the sequence of $A_k$ (meaning we pass to a subsequence, without relabeling), we can always ensure that $\| \theta \|_{L^\infty_t L^1_x(\R^{n+1})} < +\infty$.\end{proof}

\begin{remark}
	\label{rmk:boundedtotalspeedassertion}
	The sequence of solutions $\{\theta_k\}$ above demonstrates that the constant in the quantitative local boundedness property in Theorem~\ref{thm:lqlp} for drifts $b \in L^1_t L^\infty_x$ depends on the `profile' of $b$ rather than just its norm.
\end{remark}

\begin{proof}[Proof of non-borderline cases: $L^p_x L^q_t$, $\frac{3}{q} + \frac{n-1}{p} > 2$, $p \leq q$]
This construction exploits rescaled copies of $E$ and is, in a certain sense, self-similar.

\emph{1. Building blocks}. Let $(t_k) \subset (0,1)$ be an increasing sequence, with $t_k \to 1$ as $k \to +\infty$. Define $I_k = (t_k,t_{k+1})$, $R_k^2 = |I_k|$.

Let $0 \leq S \in C^\infty(0,1)$ satisfying $\int_0^1 S(t) \, dt = M$ with $M = 20n$. Define $X_k \: \R \to \R^n$ to be the solution of the ODE
\begin{equation}
	\dot X_k(t-t_k) = \underbrace{\frac{1}{|I_k|} S\left(\frac{t}{|I_k|} \right)}_{=: S_k(t-t_k)} \vec{e}_1, \quad X_k(t_k) = -10n \vec{e}_1.
\end{equation}
The `total speed' has been normalized: $\int |\dot X_k| \, dt = \int S \, dt = M$.
Define also
\begin{equation}
	\label{eq:scalingansatz}
	b_{k}(x,t) = S_k(t) U\left( \frac{x - X_k(t)}{R_k} \right)
\end{equation}
and
\begin{equation}
	E_k(x,t) = \frac{1}{R_k^n} E\left( \frac{x-X_k(t)}{R_k}, \frac{t-t_{k}}{|I_k|} \right).
\end{equation}
Then $E_{k}$ is a subsolution:
\begin{equation}
	\left( \p_t - \Delta + b_{k} \cdot \nabla \right) E_{k} \leq 0 \text{ on } \R^{n+1} \setminus \{ (X_k(t_k),t_k) \}
\end{equation}
and satisfies many of the same properties as $E_k$ in the previous construction, among which is
\begin{equation}
	E_{k}(\cdot,\tilde{t_k}) = \frac{1}{R_k^n} E\left(\frac{\cdot}{R_k},\tilde{t} \right)
\end{equation}
for some $\tilde{t_k} \in (t_k,t_{k+1})$ and $\tilde{t} \in (0,1)$. We define the solution $\theta_k$ to the PDE:
\begin{equation}
\begin{aligned}
	\left( \p_t - \Delta + b_{k} \cdot \nabla \right) \theta_k &= 0  \quad \text{ on } \R^{n+1} \setminus \{ (X_k(t_k),t_k) \}\\
	\theta_k|_{t = t_k} &= \delta_{x = X_k(t_k)},
	\end{aligned}
\end{equation}
which for short times $|t - t_k| \ll_k 1$ and times $t < t_k$ is equal to the heat kernel $\Gamma(x-X_k(t_k),t)$. The comparison principle implies
\begin{equation}
	E_{k} \leq \theta_k.
\end{equation}

\textit{2. Estimating the drift}. We now estimate the size of $b_k$. To begin, we estimate the $L^q_t L^p_x$ norms, $\frac{2}{q} + \frac{n}{p} > 2$. Using the scalings from~\eqref{eq:scalingansatz}, we have
\begin{equation}
	\label{eq:maxb}
	\max |b_k| \leq \| U \|_{L^\infty} \| S \|_{L^\infty} |I_k|^{-1}
\end{equation}
and
\begin{equation}
	\label{eq:interpguy2}
	\| b_k \|_{L^q_t L^p_x(\R^{n+1})} \les \| U \|_{L^\infty} \| S \|_{L^\infty} |I_k|^{\frac{1}{q}-1}  R_k^{\frac{n}{p}} \les R_k^{\varepsilon(p,q)} = o_{k\to+\infty}(1),
\end{equation}
since $|I_k| = R_k^2$.
Next, we estimate the $L^p_{x'} L^\infty_{x_n} L^{\tilde{q}}_t$ norm, where $\frac{2}{\tilde{q}} + \frac{n-1}{p} > 2$. We are most interested when $\tilde{q} = +\infty$ and $p = \frac{n-1}{2}-$, but it is not more effort to estimate this. Importantly, we have
\begin{equation}
	{\rm supp } \; b_k \subset B^{\R^{n-1}}_{C_n R_k} \times (-C_n,C_n) \times I_k.
\end{equation}
Using this and~\eqref{eq:maxb}, we have
\begin{equation}
	\label{eq:interpguy1}
	\| b_k \|_{L^p_{x'} L^\infty_{x_n} L^{\tilde{q}}_t(\R^{n+1})} \les \| U \|_{L^\infty} \| S \|_{L^\infty} R_k^{\frac{n-1}{p}} |I_k|^{\frac{1}{\tilde{q}}-1} \les R_k^{\varepsilon(p,\tilde{q})} = o_{k \to +\infty}(1).
\end{equation}
Interpolating between~\eqref{eq:interpguy2} and~\eqref{eq:interpguy1} with $(p,\tilde{q}) = (\frac{n-1}{2}-,\infty)$, we thus obtain
\begin{equation}
	\| b_k \|_{L^p_x L^q_t(\R^{n+1})} = o_{k \to +\infty}(1)
\end{equation}
when $\frac{3}{q} + \frac{n-1}{p} > 2$ and $p \leq q$. After `pruning' the sequence in $k$ (meaning we pass to a subsequence, without relabeling), we have
\begin{equation}
	\| b \|_{L^q_t L^p_x(\R^{n+1})} \leq \sum\| b_k \|_{L^q_t L^p_x(\R^{n+1})} < +\infty, \quad \frac{2}{q} + \frac{n}{p} > 2
\end{equation}
and
\begin{equation}
	\| b \|_{L^p_x L^q_t(\R^{n+1})} \leq \sum \| b_k \|_{L^p_x L^q_t(\R^{n+1})}  < +\infty, \quad \frac{3}{q} + \frac{n-1}{p} > 2, \quad p \leq q.
\end{equation}

\emph{3. Concluding}. The remainder of the proof proceeds as before, with the notable difference that we do not need to reorder the blocks in time. To summarize, we have
\begin{equation}
	\label{eq:lowerboundonslices2}
	\sup_t A_k \| E_{k}(\cdot,t) \|_{L^\infty(B_1)} \geq  A_k \| E_{k}(\cdot,\tilde{t_k}) \|_{L^\infty(B_1)}  \ges A_k {R_k}^{-n},
\end{equation}
where $\tilde t_k\in (t_k,t_{k+1})$ satisfies $X_{k}(\tilde{t_k})=0$, and hence,
\begin{equation}
	\label{eq:limsup2}
	\limsup_{t \to 1_-} \| \theta(\cdot,t) \|_{L^\infty(B_1)} \ges \limsup_{k \to +\infty} A_k {R_k}^{-n}
\end{equation}
To control the solution from above, we again use~\eqref{eq:thingtoprune} and choose $A_k \to 0$ as $k \to +\infty$ while maintaining that the RHS~\eqref{eq:limsup2} is infinite. By again `pruning' the sequence in $k$, we have $\| \theta \|_{L^\infty_t L^1_x(\R^{n+1})} < +\infty$. This completes the proof.
\end{proof}

\begin{remark}[An open question]
	\label{rmk:anopenquestion}
As mentioned in the introduction, we do not construct counterexamples in the endpoint cases $L^p_x L^q_t$, $\frac{3}{q} + \frac{n-1}{p} = 2$, except when $p = q = \frac{n+2}{2}$ or $(p,q) = (\frac{n-1}{2},+\infty)$ (steady example constructed above). This seems to suggest, perhaps, that local boundedness should also fail on the line between these two points, but that the counterexamples may be more subtle. It would be interesting to construct these examples. Since each `block' above is uniformly bounded in the desired spaces, we can say that, if local boundedness were to hold there, it must be depend on the `profile' of $b$ and not just its norm, as in Remark~\ref{rmk:boundedtotalspeedassertion}.
\end{remark}

\section{Upper bounds on fundamental solutions}
\label{sec:upperboundsonfundamentalsolutions}

For the Gaussian-like upper bounds on fundamental solutions, we consider a divergence-free vector field $b \in C^\infty_0(\R^n \times [0,+\infty))$ and $\alpha_0 \geq 0$, $C_0, M_0 > 0$ satisfying the following two properties:

\textbf{I}. \emph{Local boundedness}:  For each $x_0 \in \R^n$, $R > 0$, $t_0 \in \R_+$, parabolic cylinder $Q_R(x_0,t_0) = B_R(x_0) \times (t_0-R^2,t_0) \subset \R^n \times \R_+$, and Lipschitz solution $u \in W^{1,\infty}(Q_R(x_0,t_0))$, we have
\begin{equation}
	\label{eq:localboundednessforfundsol}
	\| u \|_{L^\infty(Q_{R/2}(x_0,t_0))} \leq C_0 \left( \frac{M_0}{R^{\alpha_0+2}} + \frac{1}{R^2} \right)^{\frac{n+2}{4}} \| u \|_{L^2(Q_R(x_0,t_0))}
\end{equation}

\textbf{II}. \emph{Global, revised form boundedness condition}, which we denote~\eqref{eq12.33}:
\begin{quote}
For each $x_0 \in \R^n$,  $R > 0$, open interval $I \subset \R_+$, and Lipschitz function $u \in W^{1,\infty}(B_R \times I)$, there exists a measurable set $A = A(x_0,R,b) \subset (R/2,R)$ with $|A| \geq R/4$ and satisfying
\begin{equation}
	\label{eq12.33}
	\tag{$\tilde{\text{FBC}}$}
\begin{aligned}
	& \frac{1}{|A|} \iint_{B_A \times I} \frac{|u|^2}{2} (b \cdot n) \, dx \, dt \leq \frac{M_0}{\varepsilon^{\alpha_0+1} R^{\alpha_0+2}} \iint_{B_R \times I} |u|^2 \, dx \, dt \\
	&\quad + \varepsilon \left( \frac{1}{R^2} \iint_{B_R \times I} |u|^2 \, dx \, dt + \iint_{B_R \times I} |\nabla u|^2 \, dx \, dt + \sup_{t \in I} \int_{B_R} |u(x,t)|^2 \, dx \right),
	\end{aligned}
\end{equation}
where $n$ is the unit outer normal direction and $B_A = \cup_{r \in A} \partial B_r$.
\end{quote}

Under the above conditions, we have

\begin{theorem}[Upper bounds on fundamental solutions]
	\label{thm:upperbounds}
If  $b \in C^\infty_0(\R^n \times [0,+\infty))$ is a divergence-free vector field satisfying \textbf{I} and \textbf{II}, then the fundamental solution $\Gamma = \Gamma(x,t;y,s)$ to the parabolic operator $L=\partial_t-\Delta + b \cdot \nabla$ satisfies the following upper bounds:
\begin{align}
	\label{eq:upperboundsendofpaper}
	&\Gamma(x,t;y,s) \leq C (t-s) \left( \frac{1}{t-s} + \frac{M_0}{(t-s)^{1+\frac{\alpha_0}{2}}} \right)^{\frac{n+2}{2}}\\
 &\quad\times \left[ \exp \left( - \frac{c |x-y|^2}{t-s}  \right) + \exp\left(  - \frac{ c |x-y|^{1+\frac{1}{1+\alpha_0}}}{((t-s) M_0)^{\frac{1}{1+\alpha_0}}} \right) \right]
\end{align}
for all $x,y \in \R^n$ and $0 \leq s < t < +\infty$, where $C,c > 0$ may depend on $C_0$, $\alpha_0$, and $n$.
\end{theorem}

\begin{remark}[Comments on \textbf{I} and \textbf{II}]
Notice that~\eqref{eq12.33} in~\textbf{II} bounds the \emph{outflux} of scalar through the domain. The \emph{influx}, which was bounded explicitly in~\eqref{eq:fbc}, is handled implicitly by~\textbf{I}.

The parameter $\alpha_0$ in~\eqref{eq12.33} does not track the same information as the parameter $\alpha$ in~\eqref{eq:fbc}. The constant $M_0$ in~\eqref{eq12.33} has dimensions $L^{\alpha_0}$, whereas the constant $M$ in~\eqref{eq:fbc} is dimensionless. The upper bound in~\eqref{eq:upperboundsendofpaper} is dimensionally correct.

Upon optimizing in $\varepsilon$, one discovers that~\eqref{eq12.33} is equivalent to an interpolation inequality, see the derivation in Remark~\ref{rmk:trackingparametervarepsilon0}. The above form is reflective of how the estimate is utilized below.

It would be possible to incorporate a parameter $\delta_0 \in (0,1/2)$ such that $|A| \geq \delta_0 R$, as in~\eqref{eq:fbc}.
\end{remark}

We verify that the above conditions are satisfied in the setting of the main theorems:
\begin{lemma}[Verifying $\tilde{\text{FBC}}$]
	\label{lem:verifyingtildefbc}
	Let $p,q \in [1,+\infty]$ and $b \in C^\infty_0(\R^n \times [0,+\infty))$ be a divergence-free vector field.

	1. If
	\begin{equation}
	b \in L^q_t L^p_x(\R^n \times \R_+), \quad 1 \leq \zeta_0 := \frac{2}{q} + \frac{n}{p} < 2,
	\end{equation}
	then $b$ satisfies I and II with $\alpha_0 = \frac{\zeta_0-1}{\theta_2} (=\frac 1 {\theta_2}-2)$, $M_0 = C \| b \|_{L^q_t L^p_x(\R^n \times \R_+)}^{\frac{1}{\theta_2}} + \frac{1}{4}$, $C_0 = C$, and $\theta_2 = 1-\frac{\zeta_0}{2}$. 

	2. If
	\begin{equation}
	b \in L^p_x L^q_t(\R^n \times \R_+), \quad p \leq q, \quad 1 \leq \zeta_0 := \frac{3}{q} + \frac{n-1}{p} < 2,
	\end{equation}
	then $b$ satisfies I and II with $\alpha_0 = \frac{\zeta_0-1}{\theta_2} (=\frac 1 {\theta_2}-2)$, $M_0 = C \| b \|_{L^p_x L^q_t(\R^n \times \R_+)}^{\frac{1}{\theta_2}} + \frac{1}{4}$, $C_0 = C$, and $\theta_2 = 1-\frac{\zeta_0}{2}$.
\end{lemma}
\begin{proof}
\textbf{I} follows from Proposition~\ref{pro:stuffissatisfied}, Theorem~\ref{thm:localboundedness}, and the translation and scaling symmetries. \textbf{II} follows from rescaling $R_0$ in Remark~\ref{rmk:trackingparametervarepsilon0}.
\end{proof}

\begin{proof}[Proof of Theorem~\ref{thm:upperbounds}]
By a translation of the coordinates, we need only consider $\Gamma(x_0,t_0;0,0)$ for $t_0>0$ and $x_0\in \R^n$. We adapt a method due to E. B. Davies~\cite{daviesexplicit}. Let $\psi=\psi(|x|)=\psi(r)$ be a bounded radial Lipschitz function to be specified such that $\psi=0$ when $r<|x_0|/2$ and $\psi$ is a constant when $r\ge |x_0|$. For now, we record the property $|\nabla \psi| \leq \gamma$, where $\gamma > 0$ will be specified later.

\emph{1. Weighted energy estimates}.
Let $f\in C^\infty_0(\R^n)$ and $u$ be the solution to the equation $Lu=0$ in $\R^n \times \R^+$ with initial condition $u(0,\cdot)=e^{-\psi}f$. For $t\ge 0$, denote
\begin{equation}
J(t)=\frac 1 2\int_{\R^n}e^{2\psi(y)}u^2(y,t)\,dy.	
\end{equation}
Then, by integration by parts, for $t>0$, we have
\begin{align*}
\dot J(t)&=\int_{\R^n}e^{2\psi} u \p_t u \,dy
=\int_{\R^n}e^{2\psi}u(\Delta u - b\cdot \nabla u)\,dy\\
&=-\int_{\R^n}e^{2\psi}|\nabla u|^2\,dy - 2\int_{\R^n} e^{2\psi} u \nabla u \cdot \nabla \psi \, dy - \int_{\R^n}e^{2\psi}b\cdot \nabla \frac {u^2} 2\,dy\\
&\leq -\frac{1}{2} \int_{\R^n}e^{2\psi}|\nabla u|^2\,dy + C \gamma^2 J(t)  + \int_{\R^n}e^{2\psi}b\cdot \nabla \psi u^2\,dy,
\end{align*}
where we applied Young's inequality. Hence,
\begin{equation}
\begin{aligned}
&J(t)+ \frac{1}{2} \int_0^t\int_{\R^n}e^{2\psi}|\nabla u|^2\,dy\,ds \\
&\quad \leq J(0) + C \gamma^2 \int_0^t J(s) \, ds + \int_0^t\int_{\R^n}e^{2\psi}b\cdot \nabla \psi u^2\,dy\,ds.	
\end{aligned}
\end{equation}
Now for each $t>0$, we choose $\psi'(r)=\gamma \mathbf{1}_{A}$, where $A=A(x_0,|x_0|,b)$ is the set of `good slices' from~\eqref{eq12.33} in \textbf{II}.  By~\eqref{eq12.33} applied to $e^\psi u$, we have
\begin{align*}
&J(t)+ \frac{1}{2} \int_0^t\int_{\R^n}e^{2\psi}|\nabla u|^2\,dy\,ds\\
&\leq J(0) + C \gamma^2 \int_0^t J(s) \, ds + \gamma \int_0^t\int_{B_A}e^{2\psi}b\cdot \frac{y}{|y|} u^2\,dy\,ds\\
&\le J(0)+ C \gamma^2 \int_0^t J(s) \, ds + \gamma \varepsilon |A| \left (\sup_{s\in [0,t]}J(s)+\int_0^t \int_{B_{|x_0|}}e^{2\psi}|\nabla u|^2\,dy\,ds\right)\\
&\quad+\gamma \varepsilon |A| \int_0^t \int_{B_{|x_0|}}e^{2\psi}|\nabla \psi|^2 u^2\,dy\,ds
+ \left( \frac{\gamma M_0|A|}{\varepsilon^{\alpha_0+1}(|x_0|/2)^{\alpha_0+2}} + \frac{\gamma \varepsilon|A|}{|x_0|^2} \right) \int_0^t\int_{B_{|x_0|}}
e^{2\psi}u^2\,dy\,ds.
\end{align*}
Recall that $|A|\le |x_0|/2$. We set $\varepsilon=(100 |x_0| \gamma)^{-1}$ and take the supremum in $t$ in order to absorb the 3rd term on the right-hand side, which has coefficient $\gamma \varepsilon |A|$:
\begin{align*}
&\frac 1 2 \sup_{s\in [0,t]}J(s) + \frac{1}{4} \int_0^t\int_{\R^n}e^{2\psi}|\nabla u|^2\,dy\,ds\\
&\quad \le J(0) +\left( C \gamma^2+ C M_0 \gamma^{2+\alpha_0}  + |x_0|^{-2} \right)\int_0^t J(s)\,ds.
\end{align*}
Now, by the Gronwall inequality,
\begin{equation}
                        \label{eq1.23}
J(t)\le CJ(0)e^{( C \gamma^2+ C M_0 \gamma^{2+\alpha_0} + |x_0|^{-2} ) t} .
\end{equation}

\emph{2. Duality argument}. For $s<t$, we define an operator
\begin{equation}
                \label{eq12.17}
P^\psi_{s\to t} f(x)=e^{\psi(x)}\int_{\R^n} \Gamma(x,t;y,s)e^{-\psi(y)}f(y)\,dy.
\end{equation}
By the local boundedness estimate~\eqref{eq:localboundednessforfundsol} in \textbf{I}, for any $x\in \R^n$ and $R = \sqrt{t}$,
\begin{align*}
&e^{-2\psi(x)}|P_{0\to t}^\psi f(x)|^2=|u(x,t))|^2 \leq C_0  \left( \frac{1}{t} + \frac{M_0}{t^{1+\frac{\alpha_0}{2}}} \right)^{\frac{n+2}{2}} \int_0^t\int_{B_{\sqrt{t}}(x)} u^2(y,s) \,dy\,ds.
\end{align*}
Thus, allowing $C$ to depend on $C_0$, from~\eqref{eq1.23} we have
\begin{align*}
&|P_{0\to t}^\psi f(x)|^2\le C \left( \frac{1}{t} + \frac{M_0}{t^{1+\frac{\alpha_0}{2}}} \right)^{\frac{n+2}{2}} \int_0^t \int_{B_{\sqrt{t}}(x)} e^{2\psi(x)}u^2(y,s)\,dy\,ds\\
&\le C \left( \frac{1}{t} + \frac{M_0}{t^{1+\frac{\alpha_0}{2}}} \right)^{\frac{n+2}{2}}  \sup_{y\in B_{\sqrt{t}}(x)}e^{2(\psi(x)-\psi(y))}\int_0^t J(s)\,ds\\
&\overset{\eqref{eq1.23}}{\le} C t \left( \frac{1}{t} + \frac{M_0}{t^{1+\frac{\alpha_0}{2}}} \right)^{\frac{n+2}{2}}  e^{2\gamma\min\{ \sqrt{t},|x_0|/2\}} \times  e^{( C \gamma^2+ C M_0 \gamma^{2+\alpha_0} + |x_0|^{-2} ) t} J(0).
\end{align*}
Since $J(0)=\|f\|_{L^2}^2/2$, the above inequality together with a translation in time implies that
\begin{align*}
&\|P_{s\to t}^\psi\|_{L^2\to L^\infty}^2 \le C(t-s) \left( \frac{1}{t-s} + \frac{M_0}{(t-s)^{1+\frac{\alpha_0}{2}}} \right)^{\frac{n+2}{2}} \\
&\quad\cdot e^{2\gamma\min\{\sqrt{t-s},|x_0|/2\}} e^{( C \gamma^2+ C M_0 \gamma^{2+\alpha_0} + |x_0|^{-2} ) (t-s)} J(0)  .
\end{align*}
By duality, we also have
\begin{align*}
&\|P_{s\to t}^\psi\|_{L^1\to L^2}^2 \le C(t-s) \left( \frac{1}{t-s} + \frac{M_0}{(t-s)^{1+\frac{\alpha_0}{2}}} \right)^{\frac{n+2}{2}} \\
&\quad \times e^{2\gamma\min\{ \sqrt{t-s} ,|x_0|/2\}} e^{( C \gamma^2+ C M_0 \gamma^{2+\alpha_0} + |x_0|^{-2} ) (t-s)} J(0) .
\end{align*}
Therefore,
\begin{align*}
&\|P_{0\to t}^\psi\|_{L^1\to L^\infty}
\le \|P_{0\to t/2}^\psi\|_{L^1\to L^2}\|P_{t/2\to t}^\psi\|_{L^2\to L^\infty}\\
&\le Ct \left( \frac{1}{t} + \frac{M_0}{t^{1+\frac{\alpha_0}{2}}} \right)^{\frac{n+2}{2}} e^{2\gamma\min\{\sqrt{t},|x_0|/2\}} e^{( C \gamma^2+ C M_0 \gamma^{2+\alpha_0} + |x_0|^{-2} ) t} .
\end{align*}
From \eqref{eq12.17}, the above inequality implies that for any $x,y\in \R^n$,
\begin{align*}
&e^{\psi(x)-\psi(y)}\Gamma(x,t;y,0)\\
&\le Ct \left( \frac{1}{t} + \frac{M_0}{t^{1+\frac{\alpha_0}{2}}} \right)^{\frac{n+2}{2}} e^{2\gamma\min\{\sqrt{t},|x_0|/2\}} e^{( C \gamma^2+ C M_0 \gamma^{2+\alpha_0} + |x_0|^{-2} ) t}.
\end{align*}
In particular, we have
\begin{align*}
&\Gamma(x_0,t;0,0)\\
&\le Ce^{-\psi(x_0)} t \left( \frac{1}{t} + \frac{M_0}{t^{1+\frac{\alpha_0}{2}}} \right)^{\frac{n+2}{2}}  e^{2\gamma\min\{\sqrt{t},|x_0|/2\}} e^{( C \gamma^2+ C M_0 \gamma^{2+\alpha_0} + |x_0|^{-2} ) t} \\
&\le Ct \left( \frac{1}{t} + \frac{M_0}{t^{1+\frac{\alpha_0}{2}}} \right)^{\frac{n+2}{2}} \exp \left\lbrace 2\gamma \min \Big( \sqrt{t},\frac{|x_0|}{2} \Big) - \frac{\gamma |x_0|}{4}  +  C \gamma^2t + C M_0 \gamma^{2+\alpha_0} t + |x_0|^{-2} t \right\rbrace
\end{align*}
since $\psi(0)=0$ and $\psi(x_0)\ge \gamma |x_0|/4$.

\emph{3. Optimizing $\gamma$}.
In order to have a negative term in the exponential, we first consider
\begin{equation}
	\label{eq:sillyassumption}
	|x_0| \geq 16 \sqrt{t}.
\end{equation}
Then
\begin{equation}
	2\min \Big( \sqrt{t},\frac{|x_0|}{2} \Big) - \frac{|x_0|}{4} \leq - \frac{|x_0|}{8}
\end{equation}
Notice that final term in the exponential is already controlled:
\begin{equation}
	t |x_0|^{-2} \les 1.
\end{equation}
Hence,
\begin{equation}
	\Gamma(x_0,t;0,0) \les  t \left( \frac{1}{t} + \frac{M_0}{t^{1+\frac{\alpha_0}{2}}} \right)^{\frac{n+2}{2}} \exp \left( - \frac{1}{8} \gamma |x_0| + C \gamma^2 t + C M_0 \gamma^{2+\alpha_0} t \right).
\end{equation}
The new expression inside the exponential is
\begin{equation}
	- \frac{1}{8} \gamma |x_0| + \underbrace{C \gamma^2 t}_{A} + \underbrace{C M_0 \gamma^{2+\alpha_0} t}_{B}.
\end{equation}
We divide space into an `inner region' and `outer region', and we anticipate that the rate of decay may be modified in the outer region, where $B$ dominates.

\emph{3a. Outer region}.
We consider scalings of $\gamma$ in which $- \frac{1}{8} \gamma |x_0|$ overtakes $B$. Consider
\begin{equation}
	M_0 \gamma^{2+\alpha_0} t = \varepsilon \gamma |x_0|.
\end{equation}
Then $\frac{1}{8} \gamma |x_0| \gg B$ when
\begin{equation}
	\gamma^{1+\alpha_0} = \varepsilon M_0^{-1} \frac{|x_0|}{t}
\end{equation}
and $\varepsilon \ll 1$. In this scaling, we have that $B \geq A$ ($CM_0 \gamma^{2+\alpha_0} \geq C \gamma^2$) when
\begin{equation}
	M_0 \gamma^{\alpha_0} \geq 1,
\end{equation}
or
\begin{equation}
	\label{eq:region1}
	 \frac{M_0^{\frac{1}{\alpha_0}} |x_0|}{t} \geq \varepsilon^{-1}.
\end{equation}
In this region, under the additional assumption~\eqref{eq:sillyassumption}, we have the exponential bound
\begin{equation}
	\label{eq:exponentialbound1}
	\Gamma(x_0,t;0,0) \les  t \left( \frac{1}{t} + \frac{M_0}{t^{1+\frac{\alpha_0}{2}}} \right)^{\frac{n+2}{2}} \exp\left(  - \frac{ c |x_0|^{1+\frac{1}{1+\alpha_0}}}{(t M_0)^{\frac{1}{1+\alpha_0}}} \right)
\end{equation}

\emph{3b. Inner region}. We now consider scalings of $\gamma$ in which $- \frac{1}{8} \gamma |x_0|$ overtakes the term $A$. Consider
\begin{equation}
	\gamma^2 t = \varepsilon \gamma |x_0|.
\end{equation}
Then $\frac{1}{8} \gamma |x_0| \gg A$ when
\begin{equation}
	\gamma = \varepsilon |x_0|/t
\end{equation}
 and $\varepsilon \ll 1$. In this scaling, we have that $A \geq B$ ($C \gamma^2 \geq CM_0 \gamma^{2+\alpha_0}$) when
\begin{equation}
	M_0 \gamma^{\alpha_0} \leq 1,
\end{equation}
or
\begin{equation}
	\label{eq:region2}
	M_0^{\frac{1}{\alpha_0}} \frac{|x_0|}{t} \leq \varepsilon^{-1}.
\end{equation}
In this region, under the additional assumption~\eqref{eq:sillyassumption}, we have the exponential bound
\begin{equation}
	\label{eq:exponentialbound2}
	\Gamma(x_0,t;0,0) \les  t \left( \frac{1}{t} + \frac{M_0}{t^{1+\frac{\alpha_0}{2}}} \right)^{\frac{n+2}{2}} \exp \left( - \frac{c |x_0|^2}{t}  \right).
\end{equation}

\emph{3c. Patching}. Combining~\eqref{eq:exponentialbound1} and~\eqref{eq:exponentialbound2}, we have
\begin{equation}
	\label{eq:patchedestimate}
	\Gamma(x_0,t;0,0) \les t \left( \frac{1}{t} + \frac{M_0}{t^{1+\frac{\alpha_0}{2}}} \right)^{\frac{n+2}{2}} \left[ \exp \left( - \frac{c |x_0|^2}{t}  \right) + \exp\left(  - \frac{ c |x_0|^{1+\frac{1}{1+\alpha_0}}}{(t M_0)^{\frac{1}{1+\alpha_0}}} \right) \right].
\end{equation}
under the assumption~\eqref{eq:sillyassumption}. On the other hand, when $|x_0| \leq 16 \sqrt{t}$, the fundamental solution is controlled by the Nash estimate~\cite{nash1958continuity}:
\begin{equation}
	\| \Gamma(\cdot,t;0,0) \|_{L^\infty(\R^n)} \les t^{-\frac{n}{2}},
\end{equation}
which is independent of the divergence-free drift. This contributes to the prefactor in~\eqref{eq:patchedestimate}. The proof is complete. \end{proof}

{\small
\subsubsection*{Acknowledgments} DA thanks Vladim{\'i}r {\v S}ver{\'a}k for encouraging him to answer this question and helpful discussions. DA also thanks Tobias Barker and Simon Bortz for helpful discussions, especially concerning Remark~\ref{rmk:timedependentexamples}, and their patience. DA was supported by the NDSEG Fellowship and NSF Postdoctoral Fellowship Grant No.~2002023. HD was partially supported by the Simons Foundation, Grant No.~709545, a Simons Fellowship, and the NSF under agreement DMS-2055244.
}

\bibliographystyle{alpha}
\bibliography{divfreedriftsbib}

\begin{thebibliography}{HLMP21}

\bibitem[ABES19]{simonbortz}
Pascal Auscher, Simon Bortz, Moritz Egert, and Olli Saari.
\newblock On regularity of weak solutions to linear parabolic systems with
  measurable coefficients.
\newblock {\em J. Math. Pures Appl. (9)}, 121:216--243, 2019.

\bibitem[Aro67]{aronson}
D.~G. Aronson.
\newblock Bounds for the fundamental solution of a parabolic equation.
\newblock {\em Bull. Amer. Math. Soc.}, 73:890--896, 1967.

\bibitem[BS20a]{bella2020non}
Peter Bella and Mathias Sch{\"a}ffner.
\newblock Non-uniformly parabolic equations and applications to the random
  conductance model.
\newblock {\em arXiv preprint arXiv:2009.11535}, 2020.

\bibitem[BS20b]{bellaschaffnerpqgrowth}
Peter Bella and Mathias Sch\"{a}ffner.
\newblock On the regularity of minimizers for scalar integral functionals with
  {$(p,q)$}-growth.
\newblock {\em Anal. PDE}, 13(7):2241--2257, 2020.

\bibitem[BS21]{belleschaffnerscpam}
Peter Bella and Mathias Sch\"{a}ffner.
\newblock Local boundedness and {H}arnack inequality for solutions of linear
  nonuniformly elliptic equations.
\newblock {\em Comm. Pure Appl. Math.}, 74(3):453--477, 2021.

\bibitem[Dav87]{daviesexplicit}
E.~B. Davies.
\newblock Explicit constants for {G}aussian upper bounds on heat kernels.
\newblock {\em Amer. J. Math.}, 109(2):319--333, 1987.

\bibitem[DG57]{de1957sulla}
Ennio De~Giorgi.
\newblock Sulla differenziabilit\`a e l'analiticit\`a delle estremali degli
  integrali multipli regolari.
\newblock {\em Mem. Accad. Sci. Torino. Cl. Sci. Fis. Mat. Nat. (3)}, 3:25--43,
  1957.

\bibitem[DK18]{MR3787362}
Hongjie Dong and Seick Kim.
\newblock Fundamental solutions for second-order parabolic systems with drift
  terms.
\newblock {\em Proc. Amer. Math. Soc.}, 146(7):3019--3029, 2018.

\bibitem[FG85]{fabesgarofalo}
Eugene~B. Fabes and Nicola Garofalo.
\newblock Parabolic {B}.{M}.{O}. and {H}arnack's inequality.
\newblock {\em Proc. Amer. Math. Soc.}, 95(1):63--69, 1985.

\bibitem[Fil13]{Filonov2013}
N.~Filonov.
\newblock On the regularity of solutions to the equation {$-\Delta
  u+b\cdot\nabla u=0$}.
\newblock {\em Zap. Nauchn. Sem. S.-Peterburg. Otdel. Mat. Inst. Steklov.
  (POMI)}, 410(Kraevye Zadachi Matematichesko\u{\i} Fiziki i Smezhnye Voprosy
  Teorii Funktsi\u{\i}. 43):168--186, 189, 2013.

\bibitem[FR96]{frehse1996existence}
Jens Frehse and Michael Ru{\v{z}}i{\v{c}}ka.
\newblock Existence of regular solutions to the steady {N}avier-{S}tokes
  equations in bounded six-dimensional domains.
\newblock {\em Annali della Scuola Normale Superiore di Pisa-Classe di
  Scienze}, 23(4):701--719, 1996.

\bibitem[FS18]{FilonovShilkin}
Nikolay Filonov and Timofey Shilkin.
\newblock On some properties of weak solutions to elliptic equations with
  divergence-free drifts.
\newblock In {\em Mathematical analysis in fluid mechanics---selected recent
  results}, volume 710 of {\em Contemp. Math.}, pages 105--120. Amer. Math.
  Soc., [Providence], RI, [2018] \copyright 2018.

\bibitem[FSSC98]{franchiseapioniserra}
Bruno Franchi, Raul Serapioni, and Francesco Serra~Cassano.
\newblock Irregular solutions of linear degenerate elliptic equations.
\newblock {\em Potential Anal.}, 9(3):201--216, 1998.

\bibitem[Gal11]{galdi}
G.~P. Galdi.
\newblock {\em An introduction to the mathematical theory of the
  {N}avier-{S}tokes equations}.
\newblock Springer Monographs in Mathematics. Springer, New York, second
  edition, 2011.
\newblock Steady-state problems.

\bibitem[Giu03]{giustibook}
Enrico Giusti.
\newblock {\em Direct methods in the calculus of variations}.
\newblock World Scientific Publishing Co., Inc., River Edge, NJ, 2003.

\bibitem[HL97]{hanlinbook}
Qing Han and Fanghua Lin.
\newblock {\em Elliptic partial differential equations}, volume~1 of {\em
  Courant Lecture Notes in Mathematics}.
\newblock New York University, Courant Institute of Mathematical Sciences, New
  York; American Mathematical Society, Providence, RI, 1997.

\bibitem[HLMP21]{Hofmann2021}
Steve Hofmann, Linhan Li, Svitlana Mayboroda, and Jill Pipher.
\newblock The dirichlet problem for elliptic operators having a {BMO}
  anti-symmetric part.
\newblock {\em Mathematische Annalen}, June 2021.

\bibitem[Ign14]{ignatovaslightlysupercritical}
Mihaela Ignatova.
\newblock On the continuity of solutions to advection-diffusion equations with
  slightly super-critical divergence-free drifts.
\newblock {\em Adv. Nonlinear Anal.}, 3(2):81--86, 2014.

\bibitem[IKR14]{ignatovaelliptic}
Mihaela Ignatova, Igor Kukavica, and Lenya Ryzhik.
\newblock The {H}arnack inequality for second-order elliptic equations with
  divergence-free drifts.
\newblock {\em Commun. Math. Sci.}, 12(4):681--694, 2014.

\bibitem[IKR16]{ignatovakukavicaryzhik}
Mihaela Ignatova, Igor Kukavica, and Lenya Ryzhik.
\newblock The {H}arnack inequality for second-order parabolic equations with
  divergence-free drifts of low regularity.
\newblock {\em Comm. Partial Differential Equations}, 41(2):208--226, 2016.

\bibitem[Kon07]{kontovourkis2007elliptic}
Michalis Kontovourkis.
\newblock {\em On elliptic equations with low-regularity divergence-free drift
  terms and the steady-state Navier-Stokes equations in higher dimensions}.
\newblock University of Minnesota, 2007.

\bibitem[KS19]{MR3941634}
Seick Kim and Georgios Sakellaris.
\newblock Green's function for second order elliptic equations with singular
  lower order coefficients.
\newblock {\em Comm. Partial Differential Equations}, 44(3):228--270, 2019.

\bibitem[Lie96]{liebermanbook}
Gary~M. Lieberman.
\newblock {\em Second order parabolic differential equations}.
\newblock World Scientific Publishing Co., Inc., River Edge, NJ, 1996.

\bibitem[LP19]{linhanjill}
Linhan Li and Jill Pipher.
\newblock Boundary behavior of solutions of elliptic operators in divergence
  form with a {BMO} anti-symmetric part.
\newblock {\em Comm. Partial Differential Equations}, 44(2):156--204, 2019.

\bibitem[LSUt68]{ladyzhenskaya}
O.~A. Lady\v{z}enskaja, V.~A. Solonnikov, and N.~N. Ural'~tseva.
\newblock {\em Linear and quasilinear equations of parabolic type}.
\newblock Translations of Mathematical Monographs, Vol. 23. American
  Mathematical Society, Providence, R.I., 1968.
\newblock Translated from the Russian by S. Smith.

\bibitem[LZ04]{liskevichzhang}
Vitali Liskevich and Qi~S. Zhang.
\newblock Extra regularity for parabolic equations with drift terms.
\newblock {\em Manuscripta Math.}, 113(2):191--209, 2004.

\bibitem[Moo17]{mooney}
Connor Mooney.
\newblock Finite time blowup for parabolic systems in two dimensions.
\newblock {\em Arch. Ration. Mech. Anal.}, 223(3):1039--1055, 2017.

\bibitem[Mos61]{moserelliptic}
J\"{u}rgen Moser.
\newblock On {H}arnack's theorem for elliptic differential equations.
\newblock {\em Comm. Pure Appl. Math.}, 14:577--591, 1961.

\bibitem[Mos64]{moserharnack}
J\"{u}rgen Moser.
\newblock A {H}arnack inequality for parabolic differential equations.
\newblock {\em Comm. Pure Appl. Math.}, 17:101--134, 1964.

\bibitem[Mos67]{mosererrata}
J\"{u}rgen Moser.
\newblock Correction to: ``{A} {H}arnack inequality for parabolic differential
  equations''.
\newblock {\em Comm. Pure Appl. Math.}, 20:231--236, 1967.

\bibitem[Mos71]{moserpointwise}
J.~Moser.
\newblock On a pointwise estimate for parabolic differential equations.
\newblock {\em Comm. Pure Appl. Math.}, 24:727--740, 1971.

\bibitem[Mou19]{mourgoglou2019regularity}
Mihalis Mourgoglou.
\newblock Regularity theory and {G}reen's function for elliptic equations with
  lower order terms in unbounded domains.
\newblock {\em arXiv preprint arXiv:1904.04722}, 2019.

\bibitem[MS04]{milmansemenov}
Pierre~D. Milman and Yu.~A. Semenov.
\newblock Global heat kernel bounds via desingularizing weights.
\newblock {\em J. Funct. Anal.}, 212(2):373--398, 2004.

\bibitem[Nas58]{nash1958continuity}
John Nash.
\newblock Continuity of solutions of parabolic and elliptic equations.
\newblock {\em Amer. J. math}, 80(4):931--954, 1958.

\bibitem[NUt11]{nazarov2012harnack}
A.~I. Nazarov and N.~N. Ural'~tseva.
\newblock The {H}arnack inequality and related properties of solutions of
  elliptic and parabolic equations with divergence-free lower-order
  coefficients.
\newblock {\em Algebra i Analiz}, 23(1):136--168, 2011.

\bibitem[Osa87]{osada1987}
Hirofumi Osada.
\newblock Diffusion processes with generators of generalized divergence form.
\newblock {\em J. Math. Kyoto Univ.}, 27(4):597--619, 1987.

\bibitem[QX19a]{qianxilelllq2019}
Zhongmin Qian and Guangyu Xi.
\newblock Parabolic equations with divergence-free drift in space
  {$L^\ell_tL^q_x$}.
\newblock {\em Indiana Univ. Math. J.}, 68(3):761--797, 2019.

\bibitem[QX19b]{qianxisingular2019}
Zhongmin Qian and Guangyu Xi.
\newblock Parabolic equations with singular divergence-free drift vector
  fields.
\newblock {\em J. Lond. Math. Soc. (2)}, 100(1):17--40, 2019.

\bibitem[Sak20]{sakellaris2020scale}
Georgios Sakellaris.
\newblock Scale invariant regularity estimates for second order elliptic
  equations with lower order coefficients in optimal spaces.
\newblock {\em arXiv preprint arXiv:2005.14086}, 2020.

\bibitem[Sem06]{semenovregularitytheorems}
Yu.~A. Semenov.
\newblock Regularity theorems for parabolic equations.
\newblock {\em J. Funct. Anal.}, 231(2):375--417, 2006.

\bibitem[SSvZ12]{SSSZ}
Gregory Seregin, Luis Silvestre, Vladim\'{\i}r \v{S}ver\'{a}k, and Andrej
  Zlato\v{s}.
\newblock On divergence-free drifts.
\newblock {\em J. Differential Equations}, 252(1):505--540, 2012.

\bibitem[Str88]{struwenavierstokes}
Michael Struwe.
\newblock On partial regularity results for the {N}avier-{S}tokes equations.
\newblock {\em Comm. Pure Appl. Math.}, 41(4):437--458, 1988.

\bibitem[SVZ13]{silvestre2013loss}
Luis Silvestre, Vlad Vicol, and Andrej Zlato{\v{s}}.
\newblock On the loss of continuity for super-critical drift-diffusion
  equations.
\newblock {\em Archive for Rational Mechanics and Analysis}, 207(3):845--877,
  2013.

\bibitem[Tao13]{taolocalizationcompactness}
Terence Tao.
\newblock Localisation and compactness properties of the {N}avier-{S}tokes
  global regularity problem.
\newblock {\em Anal. PDE}, 6(1):25--107, 2013.

\bibitem[Tru71]{trudinger}
Neil~S. Trudinger.
\newblock On the regularity of generalized solutions of linear, non-uniformly
  elliptic equations.
\newblock {\em Arch. Rational Mech. Anal.}, 42:50--62, 1971.

\bibitem[Wu21]{wu2021supercritical}
Bian Wu.
\newblock On supercritical divergence-free drifts.
\newblock {\em arXiv preprint arXiv:2106.02408}, 2021.

\bibitem[Zha04]{zhang2004strong}
Qi~S. Zhang.
\newblock A strong regularity result for parabolic equations.
\newblock {\em Communications in Mathematical Physics}, 244(2):245--260, 2004.

\bibitem[Zha20]{zhang2020maximum}
Xicheng Zhang.
\newblock Maximum principle for non-uniformly parabolic equations and
  applications.
\newblock {\em arXiv preprint arXiv:2012.05026}, 2020.

\bibitem[Zhi04]{zhikovuniqueness}
V.~V. Zhikov.
\newblock Remarks on the uniqueness of the solution of the {D}irichlet problem
  for a second-order elliptic equation with lower order terms.
\newblock {\em Funktsional. Anal. i Prilozhen.}, 38(3):15--28, 2004.

\end{thebibliography}

\end{document}